\documentclass[11pt]{article}

\usepackage[font=small,format=plain,labelfont=bf,up]{caption}

\usepackage[a4paper,nohead,ignorefoot,%
twoside,scale=0.85,bindingoffset=1.25cm]{geometry}

\usepackage{amssymb,amsfonts,amsmath,amsthm}

\usepackage[]{graphicx}
\usepackage{color}

\graphicspath{}
\newcommand{\texorpdfstring}[2]{#1}   
\newcommand{\url}[1]{#1} 
\usepackage{hyperref}%
\definecolor{gray}{rgb}{0.2,0.2,.2}
\hypersetup{%
  unicode=true,          
  colorlinks=true,       
  linkcolor=gray,          
  citecolor=gray      
}

\newcommand{\fspace}[1]{{\mathsf{#1}}}
\newcommand{\fspaceL}{\fspace{L}}
\newcommand{\fspaceH}{\fspace{H}}

\newcommand{\eff}{\mathrm{eff}}

\newcommand{\ol}[1]{{\overline{#1}}}

\newcommand{\Rset}{{\mathbb{R}}}
\newcommand{\Zset}{{\mathbb{Z}}}

\newcommand{\oointerval}[2]{(#1,\,#2)}%
\newcommand{\ccinterval}[2]{[#1,\,#2]}%

\newcommand{\DO}[1]{{O\at{#1}}}

\newcommand{\tdots}{{...}}%

\newlength{\mhpicDwidth}
\newlength{\mhpicDvsep}
\newlength{\mhpicDhsep}

\newlength{\mhpicPwidth}
\newlength{\mhpicPvsep}
\newlength{\mhpicPhsep}

\newlength{\mhpicWhsep}

\newcommand{\pair}[2]{{\left({#1},\,{#2}\right)}}

\newcommand{\at}[1]{{\left({#1}\right)}}
\newcommand{\nat}[1]{(#1)}
\newcommand{\bat}[1]{{\big(#1\big)}}
\newcommand{\Bat}[1]{{\Big(#1\Big)}}

\newcommand{\triple}[3]{{\left({#1},\,{#2},\,{#3}\right)}}

\newcommand{\ul}[1]{\underline{#1}}

\newcommand{\D}{\displaystyle}

\newcommand{\bigpar}{\par\quad\newline\noindent}

\newcommand{\abs}[1]{\left|{#1}\right|}
\newcommand{\babs}[1]{\big|{#1}\big|}
\newcommand{\Babs}[1]{\Big|{#1}\Big|}

\newcommand{\dint}[1]{\,\mathrm{d}#1}

\newcommand{\al}{{\alpha}}

\newcommand{\ga}{{\gamma}}

\newcommand{\ka}{{\kappa}}
\newcommand{\la}{{\lambda}}
\newcommand{\si}{{\sigma}}

\newcommand{\om}{{\omega}}

\newcommand{\calC}{\mathcal{C}}
\newcommand{\calD}{\mathcal{D}}
\newcommand{\calE}{\mathcal{E}}

\newcommand{\calK}{\mathcal{K}}

\newcommand{\calP}{\mathcal{P}}

\newcommand{\calV}{\mathcal{V}}

\theoremstyle{plain}
\newtheorem{theorem}             {Theorem}[]
\newtheorem{corollary}  [theorem]{Corollary}

\newtheorem{lemma}      [theorem]{Lemma}
\newtheorem{proposition}[theorem]{Proposition}

\newtheorem*{result*}{Main result}

\newtheorem{assumption} [theorem]{Assumption}
\numberwithin{figure}{section}
\numberwithin{table}{section}
\numberwithin{equation}{section}
\usepackage{float}
\begin{document}
%
%
%
%
\title{Mass transport in Fokker-Planck equations \\ with tilted periodic potential}
\date{\today}
\author{%
Michael Herrmann\thanks{%
Technische Universit\"at Carolo-Wilhelmina zu Braunschweig, {\tt michael.herrmann@tu-braunschweig.de}
} %
\and
Barbara Niethammer\thanks{%
Rheinische Friedrich-Wilhelms-Universit\"at Bonn, 
{\tt niethammer@iam.uni-bonn.de}
} %
} %
\maketitle
%
%
%
%
%
\begin{abstract}
We consider Fokker-Planck equations with tilted periodic potential in the subcritical regime and characterize the spatio-temporal dynamics of the partial masses in the limit of vanishing diffusion. Our convergence proof relies on suitably defined substitute masses and bounds the approximation error using the energy-dissipation relation of the
underlying Wasserstein gradient structure. In the appendix we also discuss the case of an asymmetric double-well potential and derive the corresponding limit dynamics in an elementary way.
\end{abstract}
%
%
%
\quad\newline\noindent%
\begin{minipage}[t]{0.15\textwidth}%
   Keywords: 
\end{minipage}%
\begin{minipage}[t]{0.8\textwidth}%
\emph{Fokker-Planck equations with tilted period potential}, \emph{model reduction for multiscale dynamical systems}, \emph{asymptotic analysis of singular limits}
\end{minipage}%
\medskip
\newline\noindent
\begin{minipage}[t]{0.15\textwidth}%
   MSC (2010): %
\end{minipage}%
35B25,  	
35B40,     
35Q84  	
\begin{minipage}[t]{0.8\textwidth}%
\end{minipage}%

%
\setcounter{tocdepth}{3}
\setcounter{secnumdepth}{3}{\scriptsize{\tableofcontents}}
%
%
%

\section{Introduction}
\label{sect:intro}
%
We study the Fokker-Planck equation
\begin{align}
\label{Eqn:PDE}
\tau\bat{\partial_t\varrho\triple{t}{x}{p}-\Delta_x \varrho\triple{t}{x}{p}}=\nu^2\partial_p^2\varrho\triple{t}{x}{p}+
\partial_p\bat{\bat{H^\prime\at{p}-\si}\varrho\triple{t}{x}{p}}
\end{align}
with small parameters $\tau$ and $\nu$. Here, $t$ and $x\in\Rset^n$ denote the time and 
space variable, respectively, $p\in\Rset$ stands for an internal
but scalar state variable, and the unknown $\varrho$ is supposed to be nonnegative and normalized by
\begin{align}
\label{Eqn:UnitMass}
\int\limits_{\Rset^n}\int\limits_\Rset \varrho\triple{t}{x}{p}\dint{p}\dint{x}=1\,.
\end{align}
Fokker-Planck equations arise in many branches of mathematics and the sciences, see for instance \cite{Ris89} for more background information.  We regard 
\eqref{Eqn:PDE} as a toy model to study some aspects of multi-scale analysis and model reduction for particle systems. Indeed, the PDE  \eqref{Eqn:PDE} -- which is also called  Kramers-Smoluchowski equation -- describes the evolution of the probability density of a particle 
that undergoes random walks under the influence of the potential $H$ and the force term $\si$. In the spatially homogeneous situation -- i.e., without any $x$-dependence -- the stochastic particle dynamics is governed by the over-damped Langevin or Smoluchowski equation
\begin{align}
\label{Eqn:SDE}
\tau \dint{p} = \bat{\sigma-H^\prime\at{p}}\dint{t}+\sqrt{2\nu^2}\dint{W}\,,
\end{align}
where $W$ represents a standard Wiener Process related to Brownian motion in $p$-space.
\par
\begin{figure}[ht!] %
\centering{ %
\includegraphics[width=0.9\textwidth]{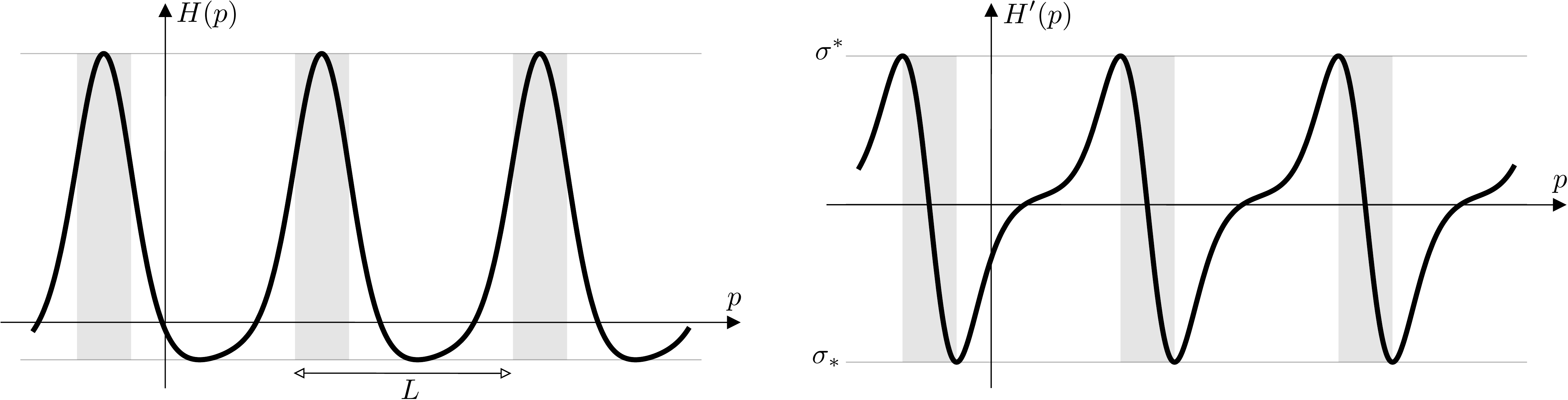}
} %
\caption{\emph{Left panel}. %
Example of an $L$-periodic potential $H$ as in Assumption \ref{Ass:Pot}. The gray boxes indicate the spinodal regions in which $H$ is concave. \emph{Right panel}. The local extrema of $H^\prime$ are denoted by $\si_*$ and $\si^*$. In this paper we always assume that the tilting parameter $\sigma$ is restricted by \eqref{Eqn:DiffRegime} so that the effective potential $H_\eff$ admits equidistant wells as illustrated in the left panel of Figure \ref{Fig:PlotEffPot}.
} %
\label{Fig:PlotPotential}
\end{figure}
In what follows we always suppose that the potential $H$ is a smooth and periodic function in $p$, see
Figure \ref{Fig:PlotPotential} for an illustration, but the particles move in the effective potential
\begin{align}
\label{Eqn:DefEffPot}
H_\eff\at{p}=H\at{p}-\si p
\end{align}
due to the presence of the tilting parameter $\si\in\Rset$, which is assumed to be independent of $\nu$. As depicted in Figure \ref{Fig:PlotEffPot}, the properties of $H_\eff$ strongly depend on the choice of $\si$, where the critical values $\si_*$ and $\si^*$ denote the global minimum and maximum of $H^\prime$ respectively.  In the supercritical regime we have either $\si<\si_*$ or $\si>\si^*$, so $H_\eff$ is either strictly increasing or decreasing. In the subcritical regime $\si_*<\si<\si^*$, however, the effective potential possesses severals wells which represent metastable traps for the stochastic particle dynamics \eqref{Eqn:SDE}. In the present paper we concentrate on the subcritical regime and study the singular limit $\nu\to0$ on the level of the Fokker-Planck equation. In particular, we derive a dynamical limit model which is still infinite-dimensional but simpler and more regular than \eqref{Eqn:PDE} as it does not involve any small parameter.
\par
\begin{figure}[ht!] %
\centering{ %
\includegraphics[width=0.45\textwidth]{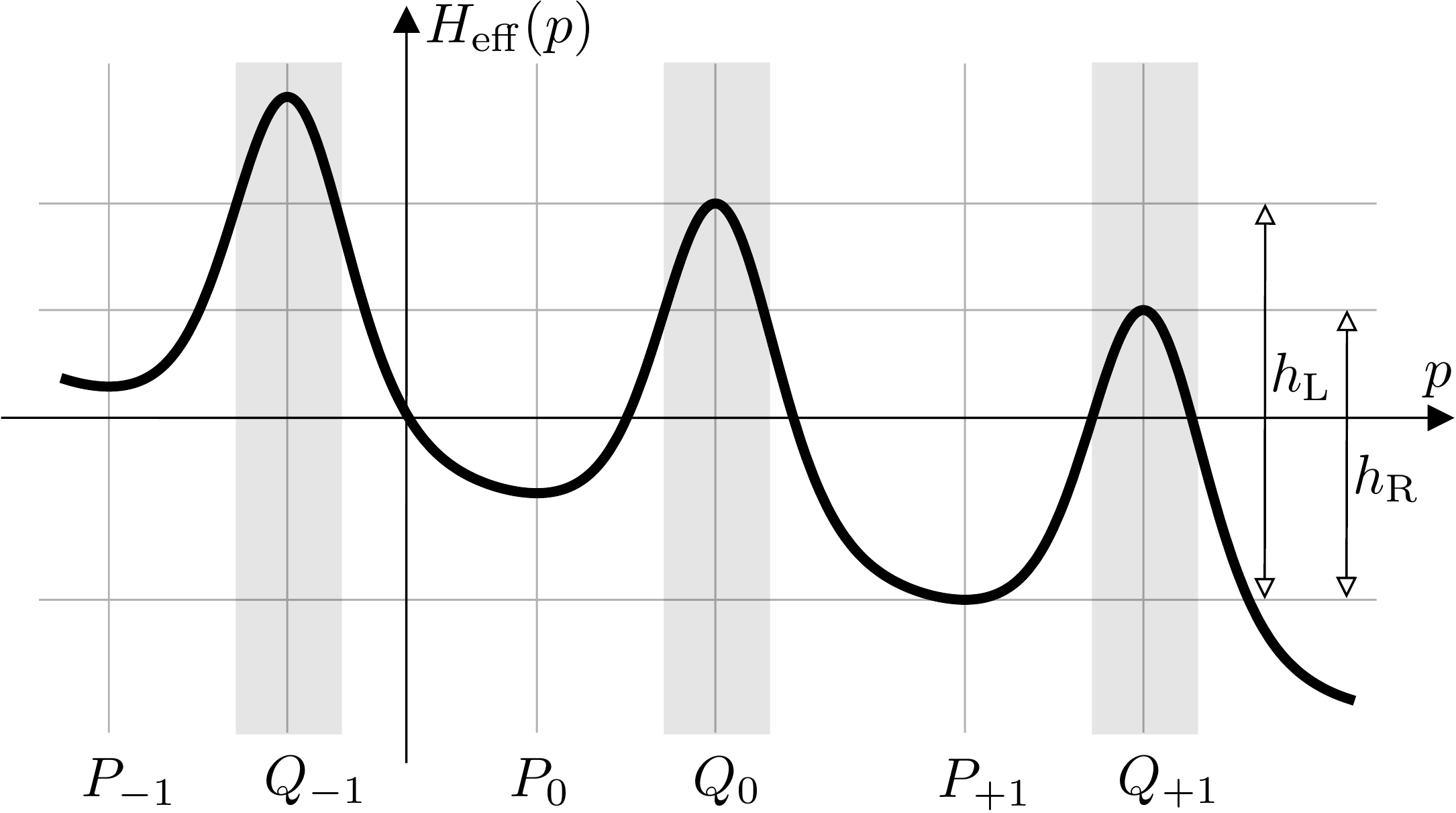} %
\hspace{0.033\textwidth} %
\includegraphics[width=0.45\textwidth]{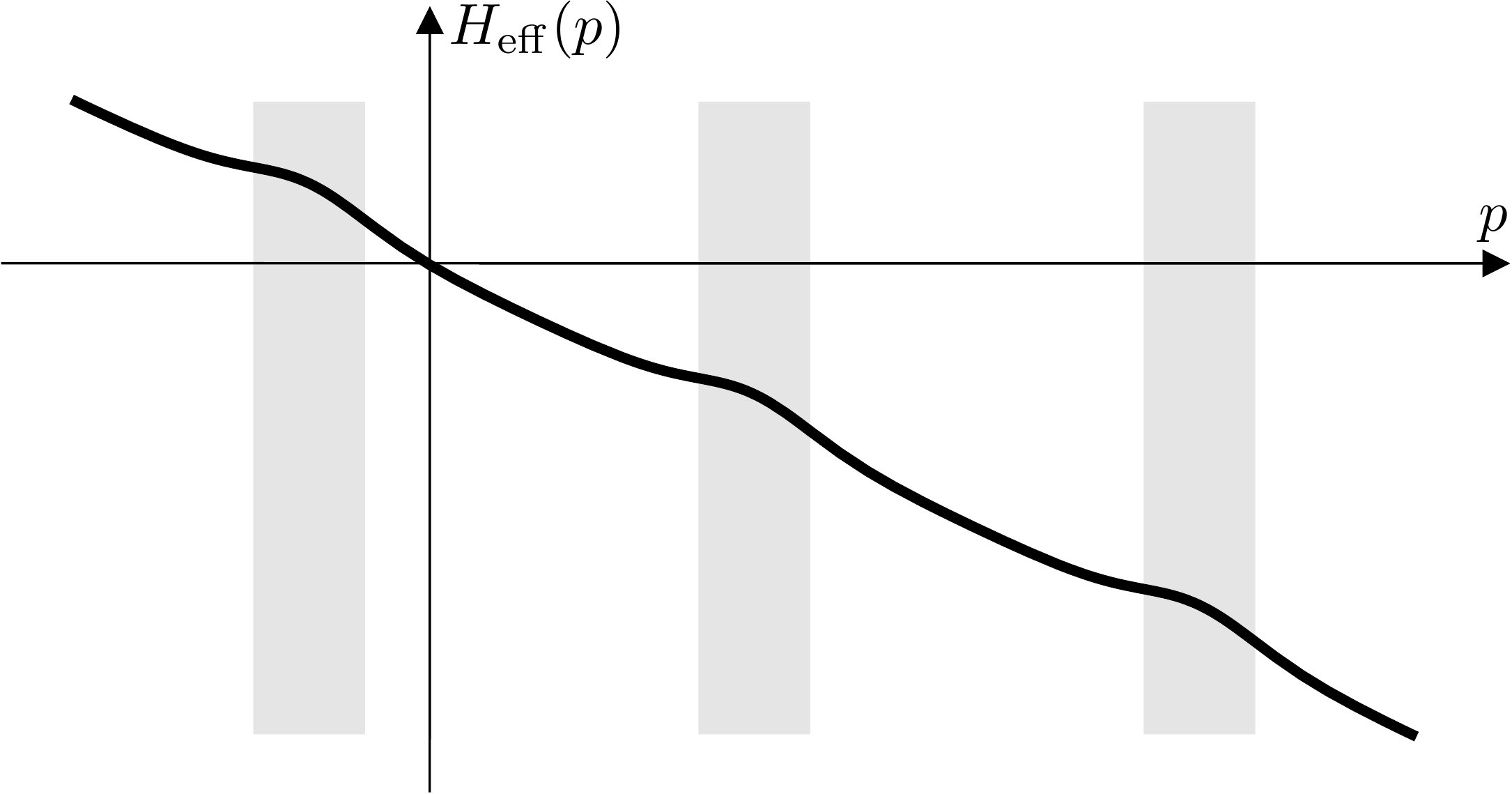}
} %
\caption{\emph{Left panel}. %
In the subcritical regime $\si_*<\si<\si^*$ studied in this paper, the effective potential $H_\eff$ from 
\eqref{Eqn:DefEffPot} admits multiple wells (here depicted for $0<\si<\si^*$) with local minima and maxima located at the positions $P_j$ and $Q_j$, respectively, where $j\in\Zset$ and $P_j<Q_j<P_{j+1}$. The gray boxes indicate the spinodal regions. \emph{Right panel}. Our results do not cover the supercritical regime $\si>\si^*$ since the effective potential has no wells anymore. The mass transfer is therefore very fast, see appendix \ref{sect:appA}.
} %
\label{Fig:PlotEffPot}
\end{figure}
Before we describe our findings and methods in more detail we emphasize that both the supercritical and the subcritical regime of \eqref{Eqn:PDE} have been studied intensively in the physics community but the main focus there is the longtime behavior of the effective velocity and the effective diffusion tensor. 
These quantities are completely determined by the first and the second $p$-moment of $\varrho$ and their averaged grow in time can be computed in many situations, see \cite{LKSG01,ReiEtAl02,SL10} for an overview (including more general models) and  \cite{HP08,LPK13,CY15} for related rigorous result. Our contribution consists in the derivation of a refined model for the limit dynamics that accounts for the mass inside of each well and in the presentation of a particular proof strategy.
%
%
\subsection{Effective mass transport in the subcritical regime}
%
%
Throughout this paper we suppose that the potential $H$ has the following properties.
\begin{assumption}[periodic part of the energy landscape]
\label{Ass:Pot}
The potential $H$ is $L$-periodic and sufficiently smooth such that
\begin{align*}
\si_* :=\min_{p\in\Rset} H^\prime\at{p}\,,\qquad 
\si^*:=\max_{p\in\Rset} H^\prime\at{p}\,,\qquad \zeta:=\sup_{p\in\Rset} \abs{H^{\prime\prime\prime}\at{p}}
\end{align*}
are well-defined. Moreover, $H^\prime$ is unimodal and non-degenerate in the sense that each critical point is a global extreme, i.e., $H^{\prime\prime}\at{p}=0$ implies $H^\prime\at{p}\in\{\si_*,\si^*\}$.
\end{assumption}
A prototypical example of Assumption \ref{Ass:Pot} is
\begin{align*}
H\at{x}=G\bat{\sin\at{x}}\,,
\end{align*}
where $G:\Rset\to\Rset$ is a smooth and strictly increasing function, and a more asymmetric example is depicted in Figure \ref{Fig:PlotPotential}.
\par
As mentioned above, we restrict our considerations to the subcritical regime. This means we fix $\si$ independent of $\nu$ with
\begin{align}
\label{Eqn:DiffRegime}
\si_*<\si<\si^*,
\end{align}
so that the effective potential from \eqref{Eqn:DefEffPot} is tilted to the right and to the left for $\si_*<\si<0$ and $0<\si<\si^*$, respectively. The constraint \eqref{Eqn:DiffRegime} guarantees that
$H_\eff$ admits an infinite number of local minima and maxima, whose positions are denoted by $P_j$ and $Q_j$, respectively. These positions depend on $\si$ but the periodicity of $H$ guarantees that $P_j=P_0+jL$ and  $Q_j=Q_0+jL$ for all $j\in\Zset$, see  Figure \ref{Fig:PlotEffPot} for an illustration.
\par
\begin{figure}[ht!] %
\centering{ %
\includegraphics[width=0.75\textwidth]{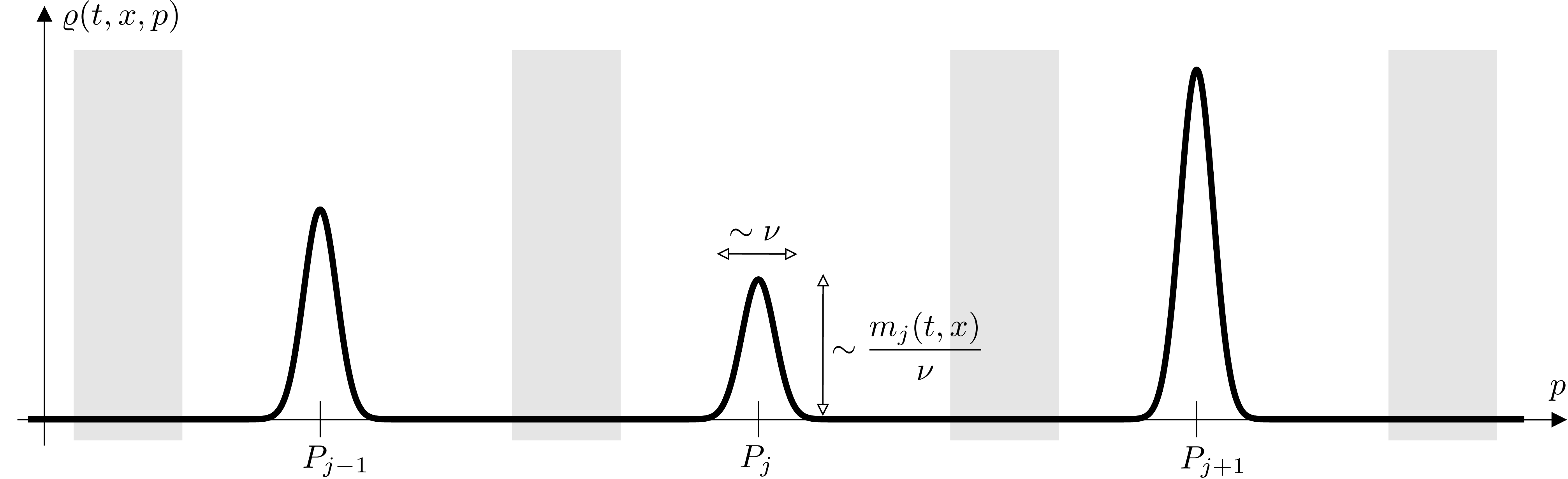}
} %
\caption{ %
Cartoon of the Fokker-Planck solution for small $0<\nu\ll1$: The function $p\mapsto\varrho\triple{t}{x}{p}$ is basically the superposition of infinitely many narrow peaks, where the $j$-th peak is localized at $p=P_j$ and carries mass $m_j\pair{t}{x}$.
These peaks do not move but exchange mass according to the limit dynamics \eqref{MR:Eqn1} or \eqref{MR:Eqn2}.
} %
\label{Fig:MassDistr}
\end{figure}
For any $j\in\Zset$ we define the partial mass
\begin{align}
\label{Eqn:DefMass}
m_j\pair{t}{x}:=\int\limits_{Q_{j-1}}^{Q_j}\varrho\triple{t}{x}{p}\dint{p}\,,
\end{align}
which quantifies at any $\pair{t}{x}$ the amount of mass that is contained in the well around the local minimum $P_j$.  The PDE \eqref{Eqn:PDE} implies that the pointwise total mass
\begin{align}
\label{Eqn:TotMass}
m\pair{t}{x}:=\sum_{j\in\Zset} m_j\pair{t}{x}
\end{align}
diffuses in $x$-space according to 
\begin{align*}
\partial_tm\pair{t}{x}-\Delta_x m\pair{t}{x}=0\,,
\end{align*}
but it remains to understand the spatio-temporal dynamics of $m_j$. This problem is well-understood on the heuristic level and the key arguments for small $\nu$ can be summarized as follows.
Due to the deterministic part in the Brownian motion it is very likely to find particles near one of the local minima. In other words, $\varrho\triple{t}{x}{\cdot}$ consists of infinitely many localized peaks and we can approximate
\begin{align}
\label{Eqn:ApproxRho}
\varrho\triple{t}{x}{p}\approx \sum_{j\in\Zset} m_j\pair{t}{x} \delta_{P_j}\at{p}
\end{align}
at least in weak* sense with Dirac distributions on the right hand side, see Figure \ref{Fig:MassDistr} for a schematic representation. The small diffusion in $p$-direction, however, guarantees 
that each peak has width of order $\DO{\nu}$ and that particles can cross the energy barriers at the local maxima of $H_\eff$ due to random fluctuations. For fixed $x$, this gives rise to a hopping process between the different wells whose characteristic time scales can be computed asymptotically by Kramers celebrated formula from \cite{Kra40}. More precisely, in the limit $\nu\to0$ the expected time for a jump to the next well on the left and on the right is given by
\begin{align*}
\tau c_{\mathrm{K}}^{-1}\exp\at{\frac{h_\mathrm{L}}{\nu^2}}\qquad\text{and}\qquad \tau c_{\mathrm{K}}^{-1}\exp\at{\frac{h_\mathrm{R}}{\nu^2}}
\end{align*}
respectively, and the periodicity of $H$ implies that the energy barriers
\begin{align*}
h_\mathrm{L}:=H_\eff\at{Q_{j-1}}-H_\eff\at{P_j}\,,\qquad h_\mathrm{R}:=H_\eff\at{Q_j}-H_\eff\at{P_j}
\end{align*}
are actually independent of $j$. Moreover, the Kramers constant
\begin{align}
\label{Eqn:Barriers}
c_{\mathrm{K}}:=\frac{\sqrt{\abs{H^{\prime\prime}\at{P_{j}}H^{\prime\prime}\at{Q_{j}}}}}{2\pi}
\end{align}
is also independent of $j$ and is the same for jumps to the left and to the right. This motivates the following choice of the time scale.
\begin{assumption}[choice of $\tau$] 
For fixed $\si$ as in \eqref{Eqn:DiffRegime} we set
\begin{align}
\label{Eqn:Kramers}
\tau:= c_{\mathrm{K}}\exp\at{-\frac{\min\{h_{\mathrm{L}},\,h_{\mathrm{R}}\}}{\nu^2}}\,,
\end{align}
where $\nu>0$ is the small but free parameter.
\end{assumption}
Due to the informal discussion about the characteristic Kramers time scales for the aforementioned hopping process we can formulate the expected limit dynamics depending on whether the value of the tilting parameter $\si$ favors transport to the left or transport to the right.
\begin{result*}[effective mass transport in the subcritical regime] 
In the limit $\nu\to0$, the partial masses evolve according to
\begin{align}
\label{MR:Eqn1}
\partial_t{m}_j\pair{t}{x}-\Delta_x m_j\pair{t}{x} = \left\{\begin{array}{lcll}
m_{j+1}\pair{t}{x}-m_j\pair{t}{x}&&\text{for}&\si_*<\si<0,\\
m_{j-1}\pair{t}{x}-m_j\pair{t}{x}&&\text{for}&0<\si<\si^*,
\end{array}\right.
\end{align}
and 
\begin{align}
\label{MR:Eqn2}
\partial_t{m}_j\pair{t}{x}-\Delta_x m_j\pair{t}{x}  = m_{j-1}\pair{t}{x}+\ka m_{j+1}\pair{t}{x}-\at{1+\ka}m_j\pair{t}{x}\quad \text{for}\quad\si=0
\,,
\end{align}
where the constant $\kappa$ depends only on the properties of $H$ and can be computed explicitly.
\end{result*}
Our goal in this paper is to justify the limit model for the partial masses rigorously in a purely analytical framework with no appeal to probabilistic techniques. It should also  be possible to justify the lattice equations \eqref{MR:Eqn1} and  \eqref{MR:Eqn2} using standard methods from stochastic analysis (such as Large Deviation Principles) but we are not aware of any reference. 
\par
We further mention that the fundamental solution to the linear limit model can be computed explicitly. For instance, assuming $0<\si<\si_*$ and that the entire initial mass is concentrated at $j=j_0$ and $x=x_0$, we readily verify that the corresponding solution to \eqref{MR:Eqn2} is given by
\begin{align}
\label{Eqn:FundSol}
m_{j}\pair{t}{x}=K_{\mathrm{heat}}\pair{t}{x-x_0}\cdot K_{\mathrm{pois}}\pair{t}{j-j_0}\,,
\end{align}
where
\begin{align*}
K_{\mathrm{heat}}\pair{t}{x}=\at{4\pi t}^{-n/2}{\exp\at{-\frac{x^2}{4t}}}\qquad\text{and}\qquad
K_{\mathrm{pois}}\pair{t}{j}=\left\{\begin{array}{ccl}0&&\text{for $j<0$}\\\D\frac{t^j\exp\at{-t}}{j!}&&\text{for $j\geq 0$}\end{array}\right.
\end{align*}
represent the heat kernel and the Poisson point process, respectively. 
%
%
\subsection{Wasserstein gradient structure and proof strategy }
%
%
The PDE \eqref{Eqn:PDE} can be regarded as a Wasserstein gradient flow on the space of probability measures. since it can be written as
\begin{align*}
\tau \partial_t \varrho = \bat{\tau^{1/2}\nu^{-1}\partial_x+\partial_p}\Bat{\varrho\,  \bat{\tau^{1/2}\nu^{-1}\partial_x+\partial_p}\,\partial_\varrho \calE}\,,
\end{align*}
where $\calE$ abbreviates the free energy of the system and $\partial_\varrho$ denotes the functional derivative.  In particular, with
\begin{align}
\label{Eqn:DefEnergy}
\calE\at{t}:=\int\limits_{\Rset^n}\int\limits_\Rset \nu^2\varrho\triple{t}{x}{p}\ln\varrho\triple{t}{x}{p}\dint{p}\dint{x}+\int\limits_{\Rset^n}\int\limits_\Rset \bat{H\at{p}-\si p}\varrho\triple{t}{x}{p}\dint{p}\dint{x}
\end{align}
we readily verify the energy balance
\begin{align}
\label{Eqn:EnergyLaw}
\tau \dot{\calE}\at{t}=-\tau\,\nu^2\, \calC\at{t}-\nu^4\,\calD\at{t}
\end{align}
by direct computations, where 
\begin{align*}
\calC\at{t}:=\int\limits_{\Rset^n}\int\limits_\Rset \frac{\Bat{\nabla_x\varrho\triple{t}{x}{p}
}^2}{\varrho\triple{t}{x}{p}}\dint{p}\dint{x}
\end{align*}
and 
\begin{align}
\label{Eqn:DefDissC}
\calD\at{t}:=\int\limits_{\Rset^n}\int\limits_\Rset
\frac{\Bat{\partial_p\varrho\triple{t}{x}{p}+ \nu^{-2}\bat{H^\prime\at{p}-\si}\varrho\triple{t}{x}{p}}^2}{\varrho\triple{t}{x}{p}}\dint{p}\dint{x}
\end{align}
yield the total dissipations due to the Brownian motion of particles in the $x$- and the $p$-direction, respectively.
\par
The variational interpretation of Fokker-Planck equations like \eqref{Eqn:PDE} has been first described in \cite{JKO97} and attracted a lot of attention during the last decades, especially for Fokker-Planck equations that admit a unique equilibrium corresponding to a global minimizer of the energy. This is, however, not true for tilted periodic potentials because the system can constantly lower its total energy by transporting mass towards $p=-\infty$ (for $\si<0$) or $p=+\infty$ (for $\si>0$), and thus there exists neither a lower bound for the energy nor a steady state for the gradient flow. The
energy-dissipation relation \eqref{Eqn:EnergyLaw} is nevertheless very useful as it provides an temporal $\fspaceL^1$-bounds for the total dissipation on each finite time interval.
\par
The gradient flow perspective has also been used to study the diffusive mass transfers in Fokker-Planck equations with double-well potential, for which the effective dynamics in the limit $\nu\to0$ is a scalar ODE that governs the mass flux though the barrier which separates the two wells. Since our work on tilted periodic potentials has much in common with this problem we discuss the recent literature in appendix \ref{sect:doublewell} and sketch how our method can be applied to the case of a double-well potential. One advantage of our approach is that it covers also asymmetric energy landscapes while most of the recent gradient flow results are restricted to even functions $H$. We also mention that potentials with finitely many wells having the same energy are studied in \cite{MZ17}. This situation shares some similarities with the untitled case $\si=0$ in our paper but the analytic techniques are rather different as they rely on a careful spectral analysis of the Fokker-Planck-operator.
\bigpar
Our approach to the asymptotic justification of the limit dynamics  consists of three main steps, which can informally be described as follows.
\begin{enumerate}
\item \emph{\ul{Effective dynamics of substitute masses:}} 
We first identify two different approximations of the partial masses such that the time derivative of the first substitute mass can be expressed in terms of the second one. In this way we obtain dynamical relations which resemble the lattice equations \eqref{MR:Eqn1} and \eqref{MR:Eqn2} up to certain error terms. The details are presented in \S\ref{sect:SubtituteMasses} and rely on the balance equations of carefully chosen moment integrals of $\varrho$ as well as the asymptotic
auxiliary results and the local equilibrium densities from \S\ref{sect:Prelim}.
\item \emph{\ul{Dissipation bounds approximation error:}} 
Another key argument is that the difference between the partial masses and their substitutes can be controlled by the Wasserstein dissipation. More precisely, we show in \S\ref{sect:ErrorEstimates}
for given $t$ that almost all mass is in fact contained in the vicinity of the local minima $p=P_j$ provided that $\calD\at{t}$ from \eqref{Eqn:DefDissC} is sufficiently small.  Similar mass-dissipation estimates have been used in \cite{HNV14}.
\item \emph{\ul{Energy balance bounds dissipation:}} We finally prove in \S\ref{sect:limit} that \eqref{Eqn:EnergyLaw} implies that $\calD$ is small in an $\fspaceL^1$-sense and hence, loosely speaking, also at most of the times $t$. This results hinges on lower bounds for $\calE\at{t}$ and hence on upper bounds for the modulus of 
\begin{align}
\label{Eqn:DefL}
\calP\at{t}:=\int\limits_{\Rset^n}\int\limits_{\Rset}p
\varrho\triple{t}{x}{p}\dint{p}\dint{x}\,,
\end{align}
but the latter can de deduced from the moment integrals for the substitute masses.
\end{enumerate}
All partial results are combined in the proof of Theorem \ref{Thm:Convergence} and imply a rather elementary justification of the lattice model for the partial masses. Moreover, the authors believe that most of the key arguments can also be  applied to other types of Fokker-Planck equations, see the appendices for first examples. Another, more challenging equation is the nonlocal variant of \eqref{Eqn:PDE}, in which $\si$ is not given a priori but enters as the time dependent Lagrangian multiplier of a dynamical constraint, see \cite{HNV12,HNV14} for a related problem.
%
%
%
%
\section{Asymptotic analysis}
\label{sect:AA}
%
%
To prove our main result from \S\ref{sect:intro} we assume from now on that
\begin{align}
\label{Eqn:DomainSigma}
0\leq\si<\si^*
\end{align}
but emphasize that the case $\si_*<\si\leq0$ can be proven along the same lines. We also denote $C$ any generic constant that is independent of $\nu$ but can depend on the potential $H$ and the choice of $\si$.
%
\subsection{Preliminaries}
\label{sect:Prelim}
%
A key quantity for our asymptotic analysis is the Gibbs function
\begin{align}
\label{Eqn:Gibbs}
\ga\at{p}:=\exp\at{\frac{-H\at{p}+\si p}{\nu^2}}\,,
\end{align}
which is illustrated in Figure \ref{Fig:Gibbs}. Notice that $\ga$ is not integrable and this reflects the lack of nontrivial steady states. This is different to other variants of the Fokker-Planck equation -- as for instance the case of a proper double-well potential as discussed in Appendix \ref{sect:doublewell} -- in which the normalization of $\ga$ defines the unique and globally attracting equilibrium. 
\begin{figure}[ht!]
\centering{
\includegraphics[width=0.95\textwidth]{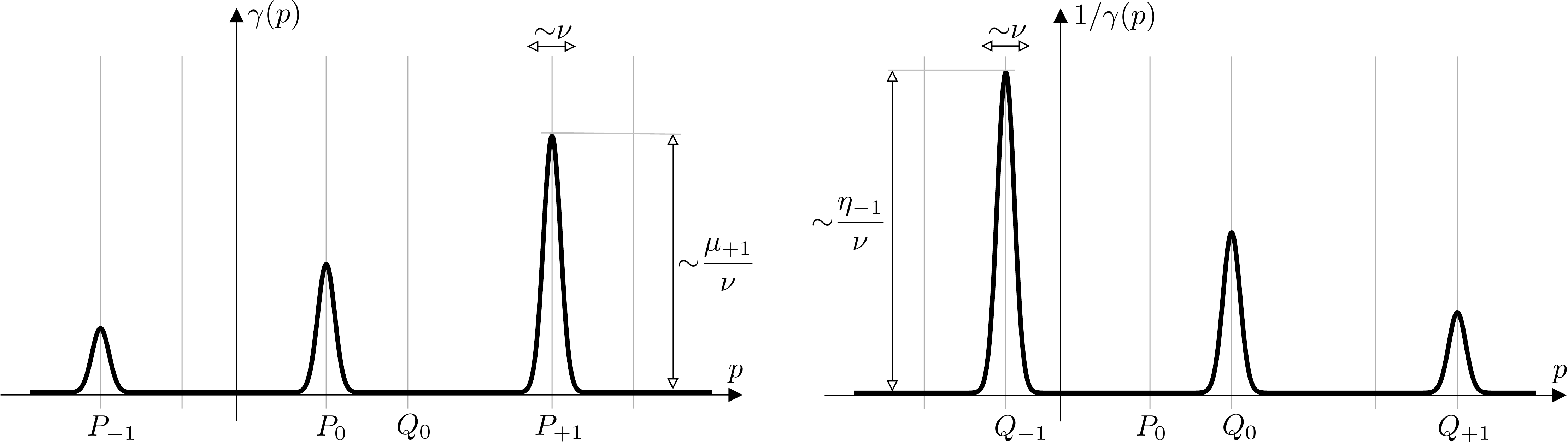} %
} %
\caption{Schematic representation of the Gibbs function $\ga$ from \eqref{Eqn:Gibbs} (\emph{left panel}) and its reciprocal (\emph{right panel}) for $0<\nu\ll1$. Each function can be approximated by a infinite superposition of equidistant peaks with width of order $\nu$, where the mass inside each peak depends exponentially on its position, see Lemma \ref{Lem:AsympIntegrals}.
} %
\label{Fig:Gibbs}
\end{figure}
\par
Our first auxiliary result characterizes the behavior of $\ga$ and $1/\ga$ in the intervals 
\begin{align}
\label{Eqn:Intervals}
J_j := \oointerval{Q_{j-1}}{Q_{j}}\qquad\text{and}\qquad K_j :=\oointerval{P_{j}}{P_{j+1}}
\end{align}
respectively, and provides a rigorous link to the exponential scaling parameter $\tau$ from the Kramers law \eqref{Eqn:Kramers}. The derivation of the latter exploits the well-known Laplace method from the theory of asymptotic integrals, see for instance \cite[section 6.4, esp. equations (6.4.1) and (6.4.35)]{BO99}.
\begin{lemma}[asymptotic integrals]
\label{Lem:AsympIntegrals}
The scalars
\begin{align}
\label{Eqn:Lem:AsympIntegrals.Eqn0}
\mu_j:=\int\limits_{J_j}\ga\at{p}\dint{p}\,,\qquad \eta_j:=\int\limits_{K_j}\frac{1}{\ga\at{p}}\dint{p}
\end{align}
satisfy
\begin{align}
\label{Eqn:Lem:AsympIntegrals.Eqn1a}
\mu_j = \mu_0 \ka^{-j}\,,\qquad 
\eta_j = \eta_{0} \ka^{+j}\,,\qquad 
\kappa := \exp\at{-\frac{\si L}{\nu^2}}\,.
\end{align}
Moreover, we have
\begin{align}
\label{Eqn:Lem:AsympIntegrals.Eqn1b}
\abs{\theta}\leq C\nu^2\,,\qquad  \theta:=\frac{\tau \mu_0\eta_{0}}{\nu^2}-1
\end{align}
for some constant $C$ which depends on $\si$ but not on  $\nu$.
\end{lemma}
\begin{proof}
The identities \eqref{Eqn:Lem:AsympIntegrals.Eqn1a} follow -- thanks to the $L$-periodicity of $H$ -- immediately from the definition in \eqref{Eqn:Gibbs} and \eqref{Eqn:Lem:AsympIntegrals.Eqn0}. Moreover, by Laplace's method we verify
\begin{align}
\label{Eqn:Lem:AsympIntegrals.PEqn1}
\mu_0=\frac{\nu\sqrt{2\pi}}{\sqrt{\abs{H^{\prime\prime}\at{P_0}}}}\exp\at{\frac{-H\at{P_0}+\si P_0}{\nu^2}}\bat{1\pm\DO{\nu^2}}
\end{align}
as well as
\begin{align*}
\eta_{0}=\frac{\nu\sqrt{2\pi}}{\sqrt{\abs{H^{\prime\prime}\at{Q_{0}}}}}\exp\at{\frac{+H\at{Q_{0}}-\si Q_{0}}{\nu^2}}\bat{1\pm\DO{\nu^2}}\,,
\end{align*}
where \eqref{Eqn:DomainSigma} ensures $H^{\prime\prime}\at{Q_0}<0<H^{\prime\prime}\at{P_0}$. We thus obtain \eqref{Eqn:Lem:AsympIntegrals.Eqn1b} thanks to the definition of $\tau$ in \eqref{Eqn:Kramers}. 
\end{proof}
Using the Gibbs function \eqref{Eqn:Gibbs} we define local equilibrium measures
\begin{align}
\label{Eqn:DefGaLoc}
\ga_j\at{p}=\mu_j^{-1}\chi_{J_j}\at{p}\ga\at{p}\,,
\end{align}
where $\chi_{J_j}$ denotes the characteristic function of the interval $J_j$. We also introduce a local relative density $w_j^2$ by
\begin{align}
\label{Eqn:DefW}
w_j^2\triple{t}{x}{p}:= \mu_j\frac{\varrho\triple{t}{x}{p}}{\ga\at{p}}\qquad \text{for}\quad p\in J_j\,,
\end{align}
where the second power on the left hand side of \eqref{Eqn:DefW} has been introduced for convenience. In terms of $w$, the partial masses from \eqref{Eqn:DefMass} can be written as
\begin{align}
\label{Eqn:FormMass}
m_j\pair{t}{x}=\int\limits_{J_j}w_j\triple{t}{x}{p}^2\ga_j\at{p}\dint{p}
\end{align}
while the dissipation due to the diffusion in $p$-space reads
\begin{align}
\label{Eqn:FormulaDiss}
\calD\at{t}=4\int\limits_{\Rset^n}D\pair{t}{x}\dint{x}
\end{align}
with
\begin{align}
\label{Eqn:FormulaDiss.2}
D\pair{t}{x}:=\sum_{j\in\Zset} D_j\pair{t}{x}
\,,\qquad   D_j\pair{t}{x}:=\int\limits_{J_j}\bat{\partial_p w_j\triple{t}{x}{p}}^2\ga_j\at{p}\dint{p}\,.
\end{align}
In particular, $m_j$ and $D_j$ are naturally related to the weighted $\fspaceL^2$- and $\fspaceH^1$-norm of $w_j$, where the weight function $\ga_j$ is a normalized and localized variant of $\ga$.
%
%
\subsection{Substitute masses and their dynamics}
\label{sect:SubtituteMasses}
%
%
As already outlined in \S\ref{sect:intro}, our asymptotic analysis is based on suitably defined substitutes to the partial masses $m_j$ from \eqref{Eqn:DefMass}. The first approximation stems from the evaluation of the relative density, i.e we set
\begin{align}
\label{Eqn:SubstMass.1}
\ol{m}_j \pair{t}{x}:= w^2_j\triple{t}{x}{P_j}\,
\end{align}
with $w_j$ as in \eqref{Eqn:DefW}. This definition is motivated by the observation that $\ga_j$ from \eqref{Eqn:DefGaLoc} is strongly localized near $P_j$ for small $\nu$ and that $w_j$ is basically constant for $p\approx P_j$ provided that the partial dissipation $D_j$ from \eqref{Eqn:FormulaDiss.2} is sufficiently small.
\par
The second substitute mass is given by

\begin{align}
\label{Eqn:SubstMass.2}
\widetilde{m}_j\pair{t}{x}:=\int\limits_\Rset \bat{\psi_{j-1}\at{p}-\psi_{j}\at{p}}\varrho\triple{t}{x}{p}\dint{p}\,,
\end{align}
where the weight function $\psi_j$ is uniquely determined by
\begin{align}
\label{Eqn:DefPsi.1}
\psi_j^\prime\at{p}:= \frac{1}{\eta_j\ga\at{p}}\quad \text{for}\quad p\in K_j
\end{align}
and 
\begin{align}
\label{Eqn:DefPsi.2}
\psi_j\at{p}=0\quad \text{for}\quad p<P_{j-1}\,,\qquad
\psi_j\at{p}=1\quad \text{for}\quad p>P_{j}\,.
\end{align}
These definitions imply
\begin{align*}
\sum_{j\in\Zset}\psi_{j-1}\at{p}-\psi_j\at{p}=1
\end{align*}
for all $p\in\Rset$ and hence 
\begin{align}
\label{Eqn:MassIdentity}
\sum_{j\in\Zset}\widetilde{m}_j\pair{t}{x}=\sum_{j\in\Zset}m_j\pair{t}{x}=m\pair{t}{x}
\end{align}
for all $t\geq0$ and any $x\in\Rset^n$.
\par
As illustrated in Figure \ref{Fig:PlotWeights}, the weight function $p\mapsto\bat{\psi_{j-1}\at{p}-\psi_{j}\at{p}}$ approximates for small $\nu>0$ the indicator function of the interval $I_j$ but the main point is that the transition layers near $Q_{j-1}$ and $Q_j$
take a particular form which enables us to compute the time derivative of $\widetilde{m}_j$ up to high accuracy.
\begin{figure}[ht!] %
\centering{ %
\includegraphics[width=0.445\textwidth]{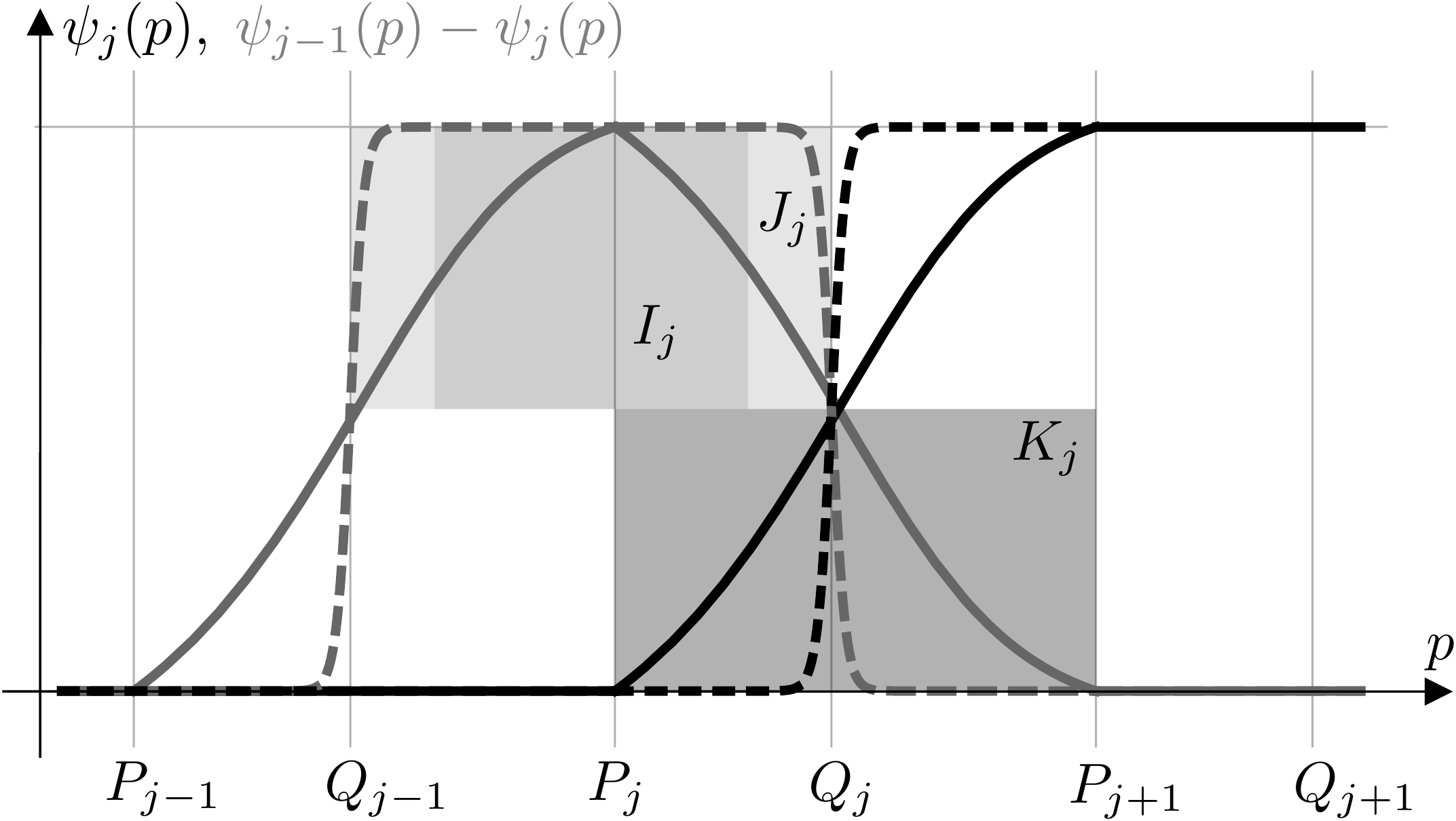} %
\hspace{0.033\textwidth} %
\includegraphics[width=0.445\textwidth]{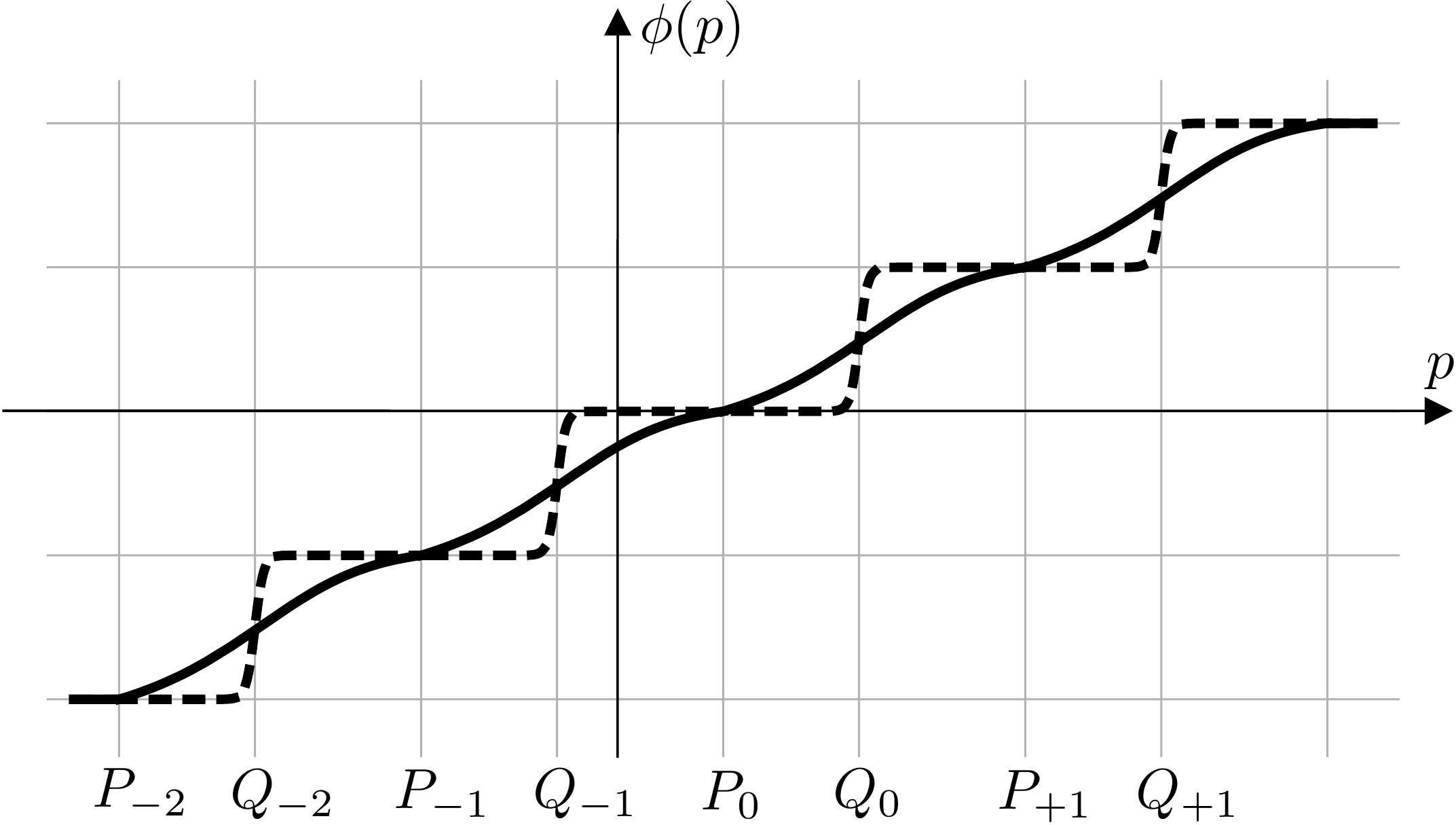} %
} %
\caption{\emph{Left panel.} Piecewise smooth moment weights as used in the definition of the substitute mass $\widetilde{m}_j\at{t}$ in \eqref{Eqn:SubstMass.2} for a small and a moderate value of $\nu$ (dashed and solid lines, respectively). The gray boxes indicate the intervals $I_j$, $J_j$, and $K_j$ from \eqref{Eqn:Intervals} and Lemma \ref{Lem:AuxResult}. \emph{Right panel.} Moment weight $\phi$ for the definition of $\calK$, see \eqref{Eqn:DefPhi} and \eqref{Eqn:DefK}, for two values of $\nu$. The mean slope of $\phi$ is $1/L$.
} %
\label{Fig:PlotWeights} %
\end{figure}
\begin{proposition}[balance of substitute masses]
\label{Prop:SubstituteMasses}
The masses from \eqref{Eqn:SubstMass.1} and \eqref{Eqn:SubstMass.2} satisfy 
\begin{align}
\label{Prop:SubstituteMasses.Eqn1}
\at{1+\theta}\Bat{\partial_t \widetilde{m}_j\pair{t}{x}-\Delta_x\widetilde{m}_j\pair{t}{x}} =\ol{m}_{j-1}\pair{t}{x}-\at{1+\ka}\,\ol{m}_j\pair{t}{x}+\ka\, \ol{m}_{j+1}\pair{t}{x}
\end{align}
where the constants $\ka$ and $\theta$ depend on $\nu$  as is Lemma \ref{Lem:AsympIntegrals}.
\end{proposition}
\begin{proof} 
By construction  -- see \eqref{Eqn:Lem:AsympIntegrals.Eqn0}, \eqref{Eqn:DefPsi.1}, and \eqref{Eqn:DefPsi.2} -- the function $\psi_j$ is continuous, piecewise smooth and satisfies on $\Rset$ the singular ODE
\begin{align*}
\nu^2\psi_j^{\prime\prime}\at{p}-\bat{H^\prime\at{p}-\si}\psi_k^\prime\at{p}=
\al_{-,\,j}\,\delta_{P_{j}}\at{p}
-\al_{+,\,j}\,\delta_{P_{j+1}}\at{p}
\end{align*}
with Dirac weights 
\begin{align*}
\al_{-,\,j}:=\nu^2\psi_j^\prime\at{P_{j}{+}0}=\frac{\nu^2}{\eta_j\ga\at{P_{j}}}\,,\qquad
\al_{+,\,j}:=\nu^2\psi_j^\prime\at{P_{j+1}{-}0}=\frac{\nu^2}{\eta_j\ga\at{P_{j+1}}}\,.
\end{align*}
Using the PDE \eqref{Eqn:PDE} and integration by parts with respect to $p$ we thus verify 
\begin{align}
\label{Prop:SubstituteMasses.PEqn1}
\begin{split}
\tau\bat{\partial_t -\Delta_x}\int_{\Rset}\psi_j\at{p}\varrho\triple{t}{x}{p}\dint{p} &=\al_{-,\,j}\varrho\triple{t}{x}{P_{j}}-\al_{+,\,j}\varrho\triple{t}{x}{P_{j+1}}
\\&=\frac{\nu^2}{\mu_j\eta_j}\ol{m}_j\pair{t}{x}-\frac{\nu^2}{\mu_{j+1}\eta_j}\ol{m}_{j+1}\pair{t}{x}
\\&=
\frac{\nu^2}{ \mu_0\eta_{0}}\bat{ \,\ol{m}_{j}\pair{t}{x}-\ka\, \ol{m}_{j+1}\pair{t}{x}}
\end{split}
\end{align}
thanks to \eqref{Eqn:Lem:AsympIntegrals.Eqn1a}, \eqref{Eqn:DefGaLoc}, \eqref{Eqn:DefW}, and \eqref{Eqn:SubstMass.1}. The claim  thus follows thanks to \eqref{Eqn:SubstMass.2} and the definition of $\theta$ in \eqref{Eqn:Lem:AsympIntegrals.Eqn1b}.
\end{proof}
Lemma \ref{Prop:SubstituteMasses.Eqn1} is at the very heart of asymptotic analysis as it provides
a dynamic relation between the different substitute masses which does not involve the small parameter $\tau$ in front of the time derivative. In particular, \eqref{Prop:SubstituteMasses.Eqn1} implies the validity of the limit model from \S\ref{sect:intro} provided that we can control the approximation errors $m_j-\bar{m}_j$ and $m_j-\widetilde{m}_j$, and this will be done below using the Wasserstein gradient structure. 
\par
A particular challenge in this context is that the energy $\calE$ is not bounded below
but decreases in $t$ since there is an effective mass transport due to the tilting of the potential. In order to estimate the decrease of $\calE$ one has to control the growth of $\calP$, but the PDE \eqref{Eqn:PDE} does not give rise to uniform bounds for $\tfrac{\dint}{\dint t}{\calP}$. To overcome this difficulty we introduce the moment 
\begin{align}
\label{Eqn:DefK}
\calK\at{t}:=\int\limits_{\Rset^n}\int\limits_{\Rset} \phi\at{p}\varrho\triple{t}{x}{p}\dint{p}\dint{x}
\end{align}
whose weight function is uniquely defined by
\begin{align}
\label{Eqn:DefPhi}
\phi^\prime\at{p}:=\sum_{j\in\Zset}\psi^\prime_j\at{p}\,,\qquad \phi\at{P_0}:=0
\end{align}
and illustrated in the right panel of  Figure \ref{Fig:PlotWeights}.
\begin{lemma}[evolution of $\calK$]
\label{Lem:EvolutionK}
We have
\begin{align*}
\at{1+\theta}\tfrac{\dint}{\dint t}\calK\at{t}=\at{1-\ka}\sum_{j\in\Zset}\,\int\limits_{\Rset^n}\ol{m}_j\pair{t}{x}\dint{x}
\end{align*}
as well as
\begin{align*}
\babs{P_0+L\,\calK\at{t}-\calP\at{t}}\leq C
\end{align*}
for some constant $C$  which does not dependent on $t$ or $\nu$.
\end{lemma}
\begin{proof} The definitions \eqref{Eqn:DefPsi.1},  \eqref{Eqn:DefPsi.1}, and \eqref{Eqn:DefPhi} yield 
\begin{align}
\label{Eqn:PropsPhi}
\phi\at{p}=\sum_{j=0}^{+\infty} \psi_j\at{p}-\sum_{j=-1}^{-\infty}\bat{1-\psi_j\at{p}}\,,
\end{align}
where the right hand side is actually a finite sum for any given $p\in\Rset$. In particular, we have
\begin{align}
\label{Eqn:BoundsPhi}
\sup_{p\in\Rset}\babs{P_0+L \phi\at{p}-p}<\infty\,,
\end{align}
and this implies the second claim. The first one follows from \eqref{Prop:SubstituteMasses.PEqn1} and \eqref{Eqn:PropsPhi} after summation over $j$ and integration with respect to $x$.
\end{proof}
%
\subsection{Asymptotic error estimates}
\label{sect:ErrorEstimates}
%
In this section we establish the key asymptotic estimates concerning the approximation of $m_j$ from \eqref{Eqn:DefMass} by the substitute masses $\ol{m}_j$ and $\widetilde{m}_j$ from \eqref{Eqn:SubstMass.1} and \eqref{Eqn:SubstMass.2}, respectively.
\begin{figure}[ht!] %
\centering{ %
\includegraphics[width=0.45\textwidth]{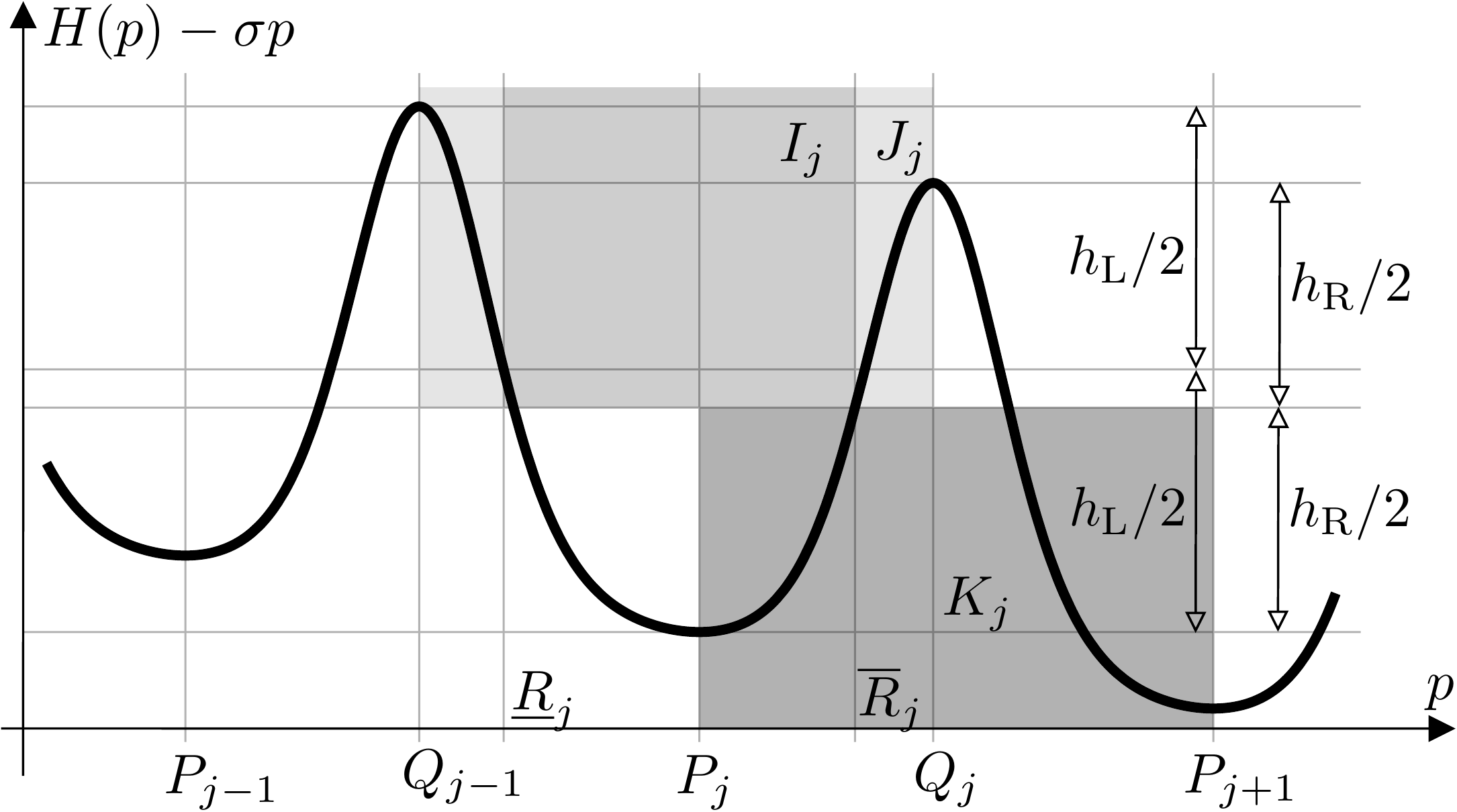} %
} %
\caption{ %
Positions $\ul{R}_j$ and $\ol{R}_j$ as used in the proof of Lemma \ref{Lem:AuxResult}, where 
$h_{\mathrm{L}}$ and $h_{\mathrm{R}}$ are the Kramers barriers from \eqref{Eqn:Barriers} and \eqref{Eqn:Kramers}, and satisfies $h_{\mathrm{R}}\leq h_{\mathrm{L}}$ thanks to $\si\geq0$.
} %
\label{Fig:AuxResult} %
\end{figure} %
\begin{lemma}[asymptotic auxiliary result]
\label{Lem:AuxResult} 
For any $j$ there exists an interval $I_j \subset J_j$ such that
\begin{align}
\notag
\int\limits_{J_j\setminus I_j}\ga_j\at{p}\leq C\nu\sqrt{\tau}\,,\qquad
\sup_{p\in I_j}\babs{\psi_{j}\at{p}}+\babs{1-\psi_{j-1}\at{p}}\leq C\nu\sqrt{\tau}
\end{align}
for some constant $C$ which depends on $\si$ but not on $\nu$.
\end{lemma}
\begin{proof} 
For any $j$ we can -- thanks to the monotonicity properties of $H^\prime$, see Assumption \ref{Ass:Pot} -- choose $\ul{R}_j$ and $\ol{R}_j$ such that
\begin{align*}
Q_{j-1}<\ul{R}_j<P_j\,,\qquad H\at{\ul{R}_j}-\si \ul{R}_j = \tfrac12\Bat{H\at{Q_{j-1}}+H\at{P_j}-\si\at{Q_{j-1}+P_j}}
\end{align*}
and
\begin{align*}
P_{j}<\ol{R}_j<Q_{j}\,,\qquad H\at{\ol{R}_j}-\si \ol{R}_j = \tfrac12\Bat{H\at{Q_{j}}+H\at{P_j}-\si\at{Q_{j}+P_j}}\,,
\end{align*}
see Figure \ref{Fig:AuxResult} for an illustration. We define
\begin{align*}
I_j:=\oointerval{\ul{R}_j}{\ol{R}_j}\,.
\end{align*}
an using the Laplace method -- compare also the asymptotic formula for $\mu_j$ in  \eqref{Eqn:Lem:AsympIntegrals.PEqn1} -- we compute

\begin{align*}
\int\limits_{P_j}^{\ol{R}_j} \ga_j\at{p}\dint{p}=C\nu\exp\at{-\frac{h_{\mathrm{L}}}{2\nu^2}}\bat{1\pm\DO{\nu}}
\end{align*}
as well as
\begin{align*}
\int\limits_{\ol{R}_j}^{P_j} \ga_j\at{p}\dint{p}=C\nu\exp\at{-\frac{h_{\mathrm{R}}}{2\nu^2}}\bat{1\pm\DO{\nu}}\,.
\end{align*}
These formulas imply the first claim due to the time scaling \eqref{Eqn:Kramers} and since $\si\geq0$ guarantees $h_{\mathrm{L}}\geq h_{\mathrm{R}}$. Finally, in view of \eqref{Eqn:DefPsi.1}+\eqref{Eqn:DefPsi.2} the second claim can be justified along the same lines.
\end{proof}
The main result in this section can be formulated as follows and controls the pointwise approximation error of the substitute masses in terms of the pointwise dissipation $D$ and the total mass $m$ from \eqref{Eqn:TotMass} and \eqref{Eqn:FormulaDiss}, respectively. 
\begin{proposition}[dissipation bounds approximation error]
\label{Prop:ErrorBounds}
We have
\begin{align*}
\sum_{j\in\Zset} \babs{m_j\pair{t}{x}-\ol{m}_j\pair{t}{x}}+\babs{m_j\pair{t}{x}-\widetilde{m}_j\pair{t}{x}}\leq C\tau^{-1/2}\nu^{2}D\pair{t}{x}+C\tau^{1/2}\nu^{-2}m\pair{t}{x}
\end{align*}
for some constant $C$ independent of $\nu$.
\end{proposition}
\begin{proof} 
Since all arguments hold pointwise in space and time, we omit both the $t$- and the $x$-dependence in all quantities.
\par
\ul{\emph{Local approximation error for $\ol{m}$}}:  By direct computations we find, using H\"olders inequality,
\begin{align}
\label{Prop:ErrorBounds.PEqn1}
\begin{split}
e_j&:=\int\limits_{J_j}\abs{w_j^2\at{p}-w_j^2\at{P_j}}\ga_j\at{p}\dint{p}\leq\int\limits_{Q_{j-1}}^{Q_{j}}\int\limits_{P_j}^{p}\babs{2w_j\at{q}\partial_pw_j\at{q}}\dint{q}\,\ga_j\at{p}\dint{p}
\\&\leq%
2\int\limits_{Q_{j-1}}^{Q_{j}}\at{\int\limits_{P_j}^{p}\frac{w_j^2\at{q}}{\ga_j\at{q}}\dint{q}}^{1/2}\at{\int\limits_{P_j}^{p}\at{\partial_pw_j\at{q}}^2\ga_j\at{q}\dint{q}}^{1/2}\ga_j\at{p}\dint{p}\,.
\end{split}
\end{align}
Since $1/\ga_j\at{p}$ is strictly increasing on the interval $\ccinterval{P_j}{Q_{j}}$ we also have
\begin{align*}
\int_{P_j}^{p}\frac{w_j^2\at{q}}{\ga_j\at{q}}\dint{q}\leq \frac{1}{\ga_j^2\at{p}}\int_{P_j}^{p}w_j^2\at{q}\ga_j\at{q}\dint{q}\leq
\frac{m_j}{\ga_j\at{p}^2}\qquad \text{for}\qquad p\in\ccinterval{P_j}{Q_{j}}
\end{align*}
due to \eqref{Eqn:FormMass}, and combining this with the analogous estimate for $p\in\ccinterval{Q_{j-1}}{P_j}$ we demonstrate that \eqref{Prop:ErrorBounds.PEqn1} can be written as
\begin{align}
\label{Prop:ErrorBounds.PEqn2}
e_j\leq C\sqrt{m_j D_j}
\end{align}
with $D_j\at{t}$ as in \eqref{Eqn:FormulaDiss.2}. This yields
\begin{align}
\label{Prop:ErrorBounds.PEqnA}
\babs{m_j-\ol{m}_j}=\Babs{\int\limits_{J_j}\at{w_j^2\at{p}-w_j^2\at{P_j}}\ga_j\at{p}\dint{p}}\leq e_j\leq C\sqrt{m_j D_j}
\end{align}
thanks to \eqref{Eqn:FormMass}, \eqref{Eqn:SubstMass.1}, and since $\int_{J_j}\ga_j\at{p}\dint{p}=1$ holds by \eqref{Eqn:Lem:AsympIntegrals.Eqn0} and \eqref{Eqn:DefGaLoc}.
\par
\ul{\emph{Local approximation error for $\widetilde{m}$}}:
With $I_j$ as in Lemma \ref{Lem:AuxResult} and in view of \eqref{Eqn:FormMass} and \eqref{Eqn:SubstMass.1}  we find
\begin{align*}
\int\limits_{J_j\setminus I_j}w_j^2\at{p}\ga_j\at{p}\dint{p}&=m_j-\int\limits_{ I_j}w_j^2\at{p}\ga_j\at{p}\dint{p}
\\&=
\bat{m_j-\ol{m}_j}\int\limits_{ I_j}\ga_j\at{p}\dint{p} +
 m_j
\int\limits_{J_j\setminus I_j}\ga_j\at{p}\dint{p}
+\int_{ I_j}\Bat{w_j^2\at{P_j}-w_j^2\at{p}}\ga_j\at{p}\dint{p}
\\&
\leq 
2e_j+m_j\int\limits_{J_j\setminus I_j}\ga_j\at{p}\dint{p}\\&
\leq C\sqrt{m_j D_j}+m_j\int\limits_{J_j\setminus I_j}\ga_j\at{p}\dint{p}\,,
\end{align*}
where we employed \eqref{Prop:ErrorBounds.PEqn2} and \eqref{Prop:ErrorBounds.PEqnA} to derive the estimates. Combining this with Lemma \ref{Lem:AuxResult} we thus obtain
\begin{align}
\label{Prop:ErrorBounds.PEqn3}
\begin{split}
\int\limits_{J_j}\psi_j\at{p}w_j^2\at{p}\ga_j\at{p}\dint{p}&=
\int\limits_{J_j\setminus I_j}\psi_j\at{p}w_j^2\at{p}\ga_j\at{p}\dint{p}+
\int\limits_{I_j}\psi_j\at{p}w_j^2\at{p}\ga_j\at{p}\dint{p}
\\&\leq
\int\limits_{J_j\setminus I_j}w_j^2\at{p}\ga_j\at{p}\dint{p}+C\nu\sqrt{\tau} m_j
\\&\leq
C\sqrt{m_j D_j}+C\nu\sqrt{\tau} m_j
\end{split}
\end{align}
and analogously
\begin{align}
\label{Prop:ErrorBounds.PEqn4}
\int\limits_{J_j}\bat{1-\psi_{j-1}\at{p}}w_j^2\at{p}\ga_j\at{p}\dint{p}\leq
C\sqrt{m_j D_j}+C\nu\sqrt{\tau} m_j\,.
\end{align}
Moreover, from \eqref{Eqn:DefMass}, \eqref{Eqn:DefW}, \eqref{Eqn:SubstMass.2}, and the piecewise definition of $\psi_j$ -- see \eqref{Eqn:DefPsi.2} -- we deduce the exact representation formula
\begin{align}
\label{Prop:ErrorBounds.PEqn5}
\begin{split}
m_j-\widetilde{m}_j &=
-\int\limits_{J_{j-1}}\psi_{j-1}\at{p}
w_{j-1}^2\at{p}\ga_{j-1}\at{p}\dint{p}\;+\int\limits_{J_{j}}\bat{1-\psi_{j}\at{p}}
w_j^2\at{p}\ga_j\at{p}\dint{p}\\&\quad\quad \quad 
+\int\limits_{J_j}\psi_{j-1}\at{p}
w_j^2\at{p}\ga_j\at{p}\dint{p}\;-\int\limits_{J_{j+1}}\bat{1-\psi_{j}\at{p}}
w_{j+1}^2\at{p}\ga_{j+1}\at{p}\dint{p}\,,
\end{split}
\end{align}
where the four terms on the right hand side represent the approximation error from the intervals
$\ccinterval{P_{j-1}}{Q_{j-1}}$, $\ccinterval{Q_{j-1}}{P_j}$, $\ccinterval{P_j}{Q_j}$, and $\ccinterval{Q_j}{P_{j+1}}$, see Figure \ref{Fig:PlotWeights}. From \eqref{Prop:ErrorBounds.PEqn5} we finally obtain the estimate
\begin{align}
\label{Prop:ErrorBounds.PEqnB}
\abs{m_j-\widetilde{m}_j }\leq C\sum_{\abs{i-j}\leq1}\Bat{\sqrt{m_{i}D_{i}}+
\nu\sqrt{\tau} m_{i}}
\end{align}
by employing \eqref{Prop:ErrorBounds.PEqn3} on both $I_{j-1}$ and $I_j$ and
\eqref{Prop:ErrorBounds.PEqn4} on $I_j$ and $I_{j+1}$.
\par
\ul{\emph{Global approximation error}}: Due to the Cauchy-Schwarz estimate
and Young's inequality for products we have
\begin{align}
\label{Prop:ErrorBounds.PEqnC}
\sum_{j\in\Zset}\sqrt{m_jD_j}\leq \Bat{\sum_{j\in\Zset}m_j}^{1/2}\Bat{\sum_{j\in\Zset}D_j}^{1/2}=\sqrt{mD}\leq \tfrac12 \tau^{1/2}\nu^{-2} m+\tfrac12\tau^{-1/2}\nu^{2}D\,,
\end{align}
so the claim follows from summing up the local estimates \eqref{Prop:ErrorBounds.PEqnA} and \eqref{Prop:ErrorBounds.PEqnB}.
\end{proof}
For completeness we also derive an approximation result for other moments of $\varrho$.
\begin{corollary}[approximation of moment integrals] 
\label{Cor:Moments}
For any smooth and bounded weight functions $v$ we have
\begin{align*}
\abs{\int\limits_\Rset v\triple{t}{x}{p}\varrho\triple{t}{x}{p}\dint{p}-\sum_{j\in\Zset}m_j\pair{t}{x}v\triple{t}{x}{P_j}}
\leq C\tau^{-1/2}\nu^{2}D\pair{t}{x}+C\nu^2m\pair{t}{x}\,,
\end{align*}
for all $t\geq0$ and all $x\in\Rset^n$, where the constant $C$ depends on $v$ but not on $\nu$.
\end{corollary}
\begin{proof} 
To ease the notation we omit again the $t$- and the $x$-dependence. Our definitions in \eqref{Eqn:DefGaLoc}, \eqref{Eqn:DefW}, and \eqref{Eqn:FormMass} imply
\begin{align*}
\int\limits_\Rset v\at{p}\varrho\at{p}\dint{p}&=
\sum_{j\in\Zset}\int\limits_{I_j} v\at{p}w_j^2\at{p}\ga_j\at{p}\dint{p}
= 
\sum_{j\in\Zset}\bat{v\nat{P_j}m_j+e_{\mathrm{a},\,j}+e_{\mathrm{b},\,j}}\,,
\end{align*}
where the error terms are given by
\begin{align*}
e_{\mathrm{a},\,j}:=\int\limits_{I_j}\bat{v\nat{p}-v\nat{P_j}}\bat{w_j^2\nat{p}-w_j^2\nat{P_j}}\ga_j\at{p}\dint{p},\qquad
e_{\mathrm{b},\,j}:=
\int\limits_{I_j}\bat{v\nat{p}-v\nat{P_j}}w_j^2\nat{P_j}\ga_j\at{p}\dint{p}\,.
\end{align*}
Similarly to the proof of Lemma \ref{Prop:ErrorBounds} -- cf. the estimates \eqref{Prop:ErrorBounds.PEqn1} and \eqref{Prop:ErrorBounds.PEqn2} -- we show
\begin{align*}
\babs{e_{\mathrm{a},\,j}}\leq C\sqrt{m_jD_j}\,,
\end{align*}
where we used that the moment weight $v$ is uniformly bounded on $I_j$, while \eqref{Eqn:SubstMass.1} and the Laplace method ensure that
\begin{align*}
\babs{e_{\mathrm{b},\,j}}\leq
\ol{m}_j\abs{\int\limits_{I_j}\bat{v\at{p}-v\nat{P_j}}\ga_j\at{p}\dint{p}}\leq C\ol{m}_j\nu^2
\end{align*}
since $\ga_j$ is localized near $p=P_j$ and because $v$ is sufficiently smooth. Thanks to  \eqref{Prop:ErrorBounds.PEqnC} the desired estimate follows after summation with respect to $j$ from Proposition \ref{Prop:ErrorBounds} and \eqref{Eqn:Kramers}. 
\end{proof}
%
%
%
\subsection{Passage to the limit \texorpdfstring{$\nu\to0$}{}}
\label{sect:limit}
%
In this section we pass to the limit $\nu$ and prove that partial masses of a solution to the Fokker-Planck equation \eqref{Eqn:PDE} converge to a solution of the limit dynamics as stated in \S\ref{sect:intro}. To this end we rely on the following assumption concerning the initial data, where
\begin{align}
\label{Eqn:DefV}
\calV\at{t}:=\int\limits_{\Rset^n}\int\limits_{\Rset}\bat{\abs{x}^2+p^2}\varrho\triple{t}{x}{p}\dint{p}\dint{x}
\end{align}
refers to the variance of $\varrho$.
\begin{assumption}[initial data] 
\label{Ass:InitialData}
The initial data are nonnegative and satisfy the normalization condition 
\begin{align*}
\int\limits_{\Rset^n}\int\limits_\Rset \varrho\triple{0}{x}{p}\dint{p}\dint{x}=1
\end{align*}
as well as the estimates
\begin{align*}
\calV\at{0}\leq C\,,\qquad  
\calE\at{0}\leq C
\end{align*}
for some constant $C$ independent of $\nu$, where the moments $\calV$ and the energy $\calE$ have been defined in \eqref{Eqn:DefV} and \eqref{Eqn:DefEnergy}, respectively.
\end{assumption}
The existence, uniqueness, and regularity of a smooth solution $\varrho$ are then guaranteed by standard results, see for instance \cite{Fri64} for a classical approach. In particular, the solution satisfies \eqref{Eqn:UnitMass} for all $t\geq0$ and this implies
\begin{align}
\label{Eqn:MassConservation}
\sum_{j\in\Zset}\int\limits_{\Rset^n} m_j\pair{t}{x}\dint{x}=1
\end{align}
Our first technical result in this section is to bound the total dissipation in the temporal $\fspaceL^1$-sense, which enables us to control the approximation errors from
Proposition \ref{Prop:ErrorBounds} in a time averaged sense. Notice that such estimates for the dissipation are not granted a priori because the energy is not bounded below 
but approaches the value $-\infty$ as $t\to\infty$. The key ingredients to our proof are the Wasserstein gradient structure as well as the estimates from Lemma \ref{Lem:EvolutionK}
for the moment $\calK$. The latter ensure that the moment $\calP$ grows nicely in time although we are not able to bound its time-derivative independently of $\nu$.
\begin{lemma}[$\fspaceL^1$-bound for the dissipation]
\label{Lem:Dissipation}
There exists a constant $C$ independent of $\nu$ such that
\begin{align*}
\int\limits_0^{T}\calD\at{t}\dint{t}\leq \tau\nu^{-4}C\at{1+T}
\end{align*}
holds for all $0<T<\infty$ and all sufficiently small $\nu>0$.
\end{lemma}
\begin{proof}
\emph{\ul{Lower bound for the energy:}} Using \eqref{Eqn:PDE} as well as integration by parts we verify
\begin{align*}
\tau\frac{\dint}{\dint{t}}\calV\at{t}&=2\at{\tau +\nu^2}-2\int\limits_{\Rset^n}\int\limits_{\Rset}p\bat{H^\prime\at{p}-\si}
\varrho\triple{t}{x}{p}\dint{p}\dint{x}\leq C\tau^{-1}+\tau\calV\at{t}\,,
\end{align*}
where we used the Young-type estimate 
\begin{align*}
2\abs{p\bat{H^\prime\at{p}-\sigma}}\leq \tau p^2+C\tau^{-1}\abs{H^\prime\at{p}-\sigma}^2\leq \tau p^2+C\tau^{-1}
\end{align*} 
as well as \eqref{Eqn:Kramers} and the conservation of mass, see \eqref{Eqn:MassConservation}. The comparison principle for scalar ODEs combined with Assumption \ref{Ass:InitialData} therefore yields
\begin{align}
\label{Lem:Dissipation.PEqn1}
\calV\at{T}\leq C\tau^{-2}\exp\at{T}\,.
\end{align}  
Since the Gaussian minimizes the convex Boltzmann entropy -- i.e., the integral of $\varrho\ln\varrho$ -- with prescribed zeroth and second moment we verify
\begin{align*}
\nu^2\int\limits_{\Rset^n}\int\limits_{\Rset}\varrho\triple{T}{x}{p}
\ln\bat{\varrho\triple{T}{x}{p}}\dint{p}\dint{x}
&\geq
C\nu^2\bat{-1-\ln{\calV\at{T}}}\\&\geq C\nu^2\bat{-1-T+\ln\tau}=C\at{-1-\nu^2T}\,,
\end{align*} 
where the first estimate stems from direct computations for Gaussian and the second one is provided by  \eqref{Lem:Dissipation.PEqn1} and the scaling law \eqref{Eqn:Kramers}. Moreover, since $H$ is bounded by Assumption \ref{Ass:Pot} we find
\begin{align*}
\int\limits_{\Rset^n}\int\limits_\Rset\bat{H\at{p}-\sigma p}\varrho\triple{T}{x}{p}\dint{p}\dint{x}\geq -C -\sigma\calP\at{t}
\end{align*}
with $\calP$ as in \eqref{Eqn:DefL}, while the properties of $\phi$ and $\calK$ in 
\eqref{Eqn:DefK} and \eqref{Eqn:BoundsPhi} imply
\begin{align*}
\babs{\calP\at{T}-L\,\calK\at{T}}\leq C\,.
\end{align*}
In summary, we have 
\begin{align}
\label{Lem:Dissipation.PEqn2}
\calE\at{T}\geq -C\at{1+\nu^2T+\babs{\calK\at{T}}}\,.
\end{align}
\emph{\ul{Upper bound for the dissipation:}} 
The energy balance \eqref{Eqn:EnergyLaw} provides
\begin{align}
\label{Lem:Dissipation.PEqn3}
0\leq \int\limits_{0}^T\bat{\calD\at{t}+\tau \,\nu^{-2}\,\calC\at{t}}\dint{t}\leq \tau\nu^{-4}\bat{\calE\at{0}-\calE\at{T}}
\end{align}
and Lemma \ref{Lem:EvolutionK} guarantees
\begin{align*}
\babs{\calK\at{T}-\calK\at{0}}\leq\int\limits_0^T\sum_{j\in\Zset}\,\int\limits_{\Rset^n}\ol{m}_j\pair{t}{x}\dint{x}\dint{t}=T+\int\limits_0^T\sum_{j\in\Zset}\,\int\limits_{\Rset^n}\babs{m_j\pair{t}{x}-\ol{m}_j\pair{t}{x}}\dint{x}\dint{t}\,,
\end{align*}
where we used that total mass is conserved due to \eqref{Eqn:MassConservation}. Exploiting Proposition \ref{Prop:ErrorBounds} and the conservation of mass we further get
\begin{align}
\label{Lem:Dissipation.PEqn4}
\begin{split}
\sum_{j\in\Zset}\,\int\limits_{\Rset^n}\babs{m_j\pair{t}{x}-\ol{m}_j\pair{t}{x}}\dint{x}
&\leq C\int\limits_{\Rset^n}\at{\tau^{-1/2}\nu^2D\pair{t}{x}+\tau^{1/2}\nu^{-2} m\pair{t}{x}}\dint{x}
\\&\leq \tau^{-1/2}\nu^{2}\calD\at{t}+\tau^{1/2}\nu^{-2}\,.
\end{split}
\end{align}
Assumption \ref{Ass:InitialData} ensures $\calE\at{0}+\abs{\calK\at{0}}\leq C$, so combining \eqref{Lem:Dissipation.PEqn2}, \eqref{Lem:Dissipation.PEqn3}, and \eqref{Lem:Dissipation.PEqn4} we arrive at
\begin{align*}
\int\limits_{0}^T\calD\at{t}\dint{t}\leq C\tau\nu^{-4}\at{1+T+\int\limits_0^T\Bat{\tau^{-1/2}\nu^{2}\calD\at{t}+\tau^{1/2}\nu^{-2}}\dint{t}}\,.
\end{align*}
The thesis now follows from rearranging terms and since $\tau$ is exponentially small in $\nu$ according to Kramers' law \eqref{Eqn:Kramers}.
\end{proof}
We are now able to prove our main result on the dynamics in the small diffusivity limit $\nu\to0$. To ease the notation we restrict ourselves to the case $0<\si<\si_*$ but emphasize that all arguments can be easily adapted to the cases $\si_*<\si<0$ and $\si=0$.
\begin{theorem}[limit dynamics]
\label{Thm:Convergence}
For $0<\si<\si_*$ and fixed $0<T<\infty$ we have
\begin{align}
\label{Thm:Convergence.Eqn1}
\sum_{j\in\Zset}\int\limits_0^T\int\limits_{\Rset^n}\babs{\breve{m}_j\pair{t}{x}-m_j\pair{t}{x}}\dint{x}\dint{t}\leq C\nu^2\at{1+T^2}\,,
\end{align}
where $\breve{m}$ denotes the unique solution to the initial value problem 
\begin{align}
\label{Thm:Convergence.Eqn2}
\partial_t \breve{m}_j\pair{t}{x}-\Delta_x\breve{m}_j\pair{t}{x}=\breve{m}_{j-1}\pair{t}{x}-\breve{m}_j\pair{t}{x}\,,\qquad
\breve{m}_j\pair{0}{x}=\widetilde{m}_j\pair{0}{x}
\end{align}
and depends on $\nu$ via the initial data.
\end{theorem}
\begin{proof} 
\emph{\ul{Error terms and bounds}}: Proposition \eqref{Prop:SubstituteMasses} provides
\begin{align*}
\partial_t \widetilde{m}_j\pair{t}{x}-\Delta_x\widetilde{m}_j\pair{t}{x}=\widetilde{m}_{j-1}\pair{t}{x}-\widetilde{m}_j\pair{t}{x}+\frac{f_j\pair{t}{x}+g_j\pair{t}{x}+h_j\pair{t}{x}}{1+\theta}\,,
\end{align*}
where the error terms on the right hand are given by
\begin{align*}
f_j:=\bat{\ol{m}_{j-1}-\widetilde{m}_{j-1}}-\,\bat{\ol{m}_j-\widetilde{m}_{j}}
\end{align*}
as well as
\begin{align*}
g_j:=\ka\bat{\ol{m}_{j+1}-\widetilde{m}_{j+1}}-\ka\bat{\ol{m}_{j}-\widetilde{m}_{j}}\,.
\end{align*}
and
\begin{align*}
h_j:=\ka\bat{\widetilde{m}_{j+1}-\widetilde{m}_j}+\theta\bat{ \widetilde{m}_j-\widetilde{m}_{j-1}}
\end{align*}
From Proposition \ref{Prop:ErrorBounds}, Lemma \ref{Lem:Dissipation} and \eqref{Eqn:Lem:AsympIntegrals.Eqn1a} we infer the estimate
\begin{align*}
\sum_{j\in\Zset}\int\limits_0^T\int\limits_{\Rset^n} \babs{f_j\pair{t}{x}}+\babs{g_j\pair{t}{x}}\dint{x}\dint{t}\leq C\int\limits_0^T\Bat{\tau^{-1/2}\nu^2\calD\at{t}+\tau^{1/2}\nu^{-2}}\dint{t}\leq C\tau^{1/2}\nu^{-2}\at{1+T}\,,
\end{align*}
while the conservation of mass combined with \eqref{Eqn:MassIdentity} gives
\begin{align*}
\sum_{j\in\Zset}\int\limits_0^T\int\limits_{\Rset^n} \babs{h_j\pair{t}{x}}\dint{x}\dint{t}\leq
2\at{\ka+\theta}\sum_{j\in\Zset}\int\limits_0^T\int\limits_{\Rset^n}\widetilde{m}_j\pair{t}{x}
\dint{x}\dint{t}=C\at{\ka+\theta}T\,.
\end{align*}
\emph{\ul{Properties of the limit dynamics}}: The linear limit model gives rise to well-defined  semigroup which is non-expansive with respect to the natural $\fspaceL^1$-norm (sums over $j$ and integrals with respect to  $x$) as it preserves the positivity and conserves mass, see also the explicit formula for the fundamental solution in \eqref{Eqn:FundSol}. We can therefore apply Duhamel's principle to the difference $\breve{m}-\widetilde{m}$ and obtain
\begin{align*}
\sum_{j\in\Zset}\int\limits_{\Rset^n}\, \babs{\breve{m}_j\pair{t}{x}-\widetilde{m}_j\pair{t}{x}}\dint{x}\leq\int_0^t e\at{s}\dint{s}\,,
\end{align*}
where 
\begin{align*}
e\at{t}:=\at{1+\theta}^{-1}\sum_{j\in\Zset}\,\int\limits_{\Rset^n}\Bat{\babs{f_j\pair{t}{x}}+\babs{g_j\pair{t}{x}
}+\babs{h_j\pair{t}{x}}}\dint{x}\,.
\end{align*}
\emph{\ul{Concluding arguments}}: All partial results derived so far imply
\begin{align*}
\sum_{j\in\Zset}\,\int\limits_0^T\int\limits_{\Rset^n} \babs{m_j\pair{t}{x}-\widetilde{m}_j\pair{t}{x}}\dint{x}\dint{t}&\leq
\int\limits_0^T \int\limits_0^t e\at{s}\dint{s}\dint{t}\leq T\int\limits_0^Te\at{t}\dint{t}
\\&\leq  C \at{T+T^2}\at{\tau^{1/2}\nu^{-2}+\ka+\theta}
\leq C \at{T+T^2}\nu^2
\end{align*} 
where the last estimate holds thanks to the scaling laws for $\tau$, $\kappa$ and $\theta$, see \eqref{Eqn:Kramers},  \eqref{Eqn:Lem:AsympIntegrals.Eqn1a}, and \eqref{Eqn:Lem:AsympIntegrals.Eqn1b}. The thesis now follows since 
\begin{align*}
\sum_{j\in\Zset}\,\int\limits_0^T\int\limits_{\Rset^n} \babs{m_j\pair{t}{x}-\widetilde{m}_j\pair{t}{x}}\dint{x}\dint{t}\leq 
C\tau^{1/2}\nu^{-2}\at{1+T}
\end{align*}
is another consequence of Proposition \ref{Prop:ErrorBounds} and Lemma \ref{Lem:Dissipation}. 
\end{proof}
The rather large  error in \eqref{Thm:Convergence.Eqn1} stems from the estimate $\abs{\theta}=\DO{\nu^2}$. If we replaced the time scaling \eqref{Eqn:Kramers} by the refined but less explicit law
\begin{align*}
\tau = \frac{\nu^2}{\mu_0\eta_0}=c_{\mathrm{K}}\exp\at{-\frac{\min\{h_{\mathrm{L}},\,h_{\mathrm{R}}\}}{\nu^2}}\Bat{1+\DO{\nu^2}}
\end{align*}  
with 
$\nu$-dependent integrals $\mu_0$, $\eta_0$ as in \eqref{Eqn:Lem:AsympIntegrals.Eqn1a}, the approximation error would be of order $\DO{\ka+\tau^{1/2}\nu^{-2}}$ and hence exponentially small in $\nu$. Notice also that the initial data for $\breve{m}$ in \eqref{Thm:Convergence.Eqn2} are defined in terms of $\widetilde{m}_j\pair{0}{x}$ instead of  $m_j\pair{0}{x}$. The difference $\sum_{j\in\Zset}\int_{\Rset^n}\abs{\widetilde{m}_j\pair{0}{x}-m_j\pair{0}{x}}\dint{x}$ is 
small for sufficiently nice initial data -- for instance, if the initial dissipation $\calD\at{0}$ is small -- and can only be large if a non-negligible amount of the initial mass is concentrated in the $\nu$-vicinity of the local maxima of the effective potential, i.e., near the $Q_j$'s.  In the latter case a fast transient dynamics can/will produce rapid changes in the masses $m_j$ while the substitute masses $\widetilde{m}_j$ still evolve quite regularly according to the limit dynamics. 
\par
We finally mention that the combination of Theorem \ref{Thm:Convergence} and Corollary \ref{Cor:Moments} implies  the time-dependent probability measure $\varrho$ can in fact be approximated as in \eqref{Eqn:ApproxRho}. Moreover, adapting the arguments from in the proof of Corollary \ref{Cor:Moments} we also verify
\begin{align*}
\calP\at{t}\approx \sum_{j\in\Zset}\;{P_j}\int\limits_{\Rset^n}{m}_j\pair{t}{x}\dint{x}\,,\qquad \calV\at{t}\approx \sum_{j\in\Zset}\;{P_j}^2\int\limits_{\Rset^n}{m}_j\pair{t}{x}\dint{x}
\end{align*}
where the error terms can be bounded for $0\leq t\leq T$ explicitly in terms of $\nu$ and $\int_0^T\calD\at{t}\dint{t}$. 
\appendix
%
\section{Mass transport in the supercritical regime}\label{sect:appA}
%
In this appendix we show that appropriately defined moment integrals are also useful in the 
supercritical (or ballistic) case $\si\notin\ccinterval{\si_*}{\si^*}$, in which the effective potential for \eqref{Eqn:PDE} has no local extrema, see the right panel of Figure \ref{Fig:PlotEffPot}. For simplicity we restrict our considerations to 
\begin{align*}
\si>\si^*
\end{align*}
and show that the large-time evolution of the first $p$-moment can be deduced from the balance law of a substitute moment. In this way we recover the well-known linear grow relation for $\calP\at{t}$, see for instance \cite{ReiEtAl02}, which reveals that the natural choice for the ballistic time scale is $\tau=1$.
\begin{proposition}[center of mass in the supercritical regime]
%
There exists a constant $C$ such that
\begin{align*}
\tau \babs{\calP\at{t}-\la\, t}\leq C\at{1+\tau+\nu^2t}\qquad \text{with}\qquad \frac{1}{\la} :=\frac{1}{L}\int\limits_{0}^L \frac{\dint{p}}{H^\prime\at{p}-\si}
\end{align*}
holds for any $t$ and all sufficiently small $\nu>0$.
\end{proposition}
\begin{proof}
We define a moment weight $\psi$ by the ODE initial value problem
\begin{align}
\label{Prop:Ballistic.PEqn0}
\nu^2\psi^{\prime\prime}\at{p}=\bat{H^\prime\at{p}-\si}\psi^\prime\at{p}+1\,,\qquad \psi^\prime\at{0}=c\,,\qquad \psi\at{0}=0\,,
\end{align}
where $c$ will be chosen below, and using integration by parts we infer from \eqref{Eqn:PDE} the identity
\begin{align}
\label{Prop:Ballistic.PEqn1}
\begin{split}
\tau\frac{\dint{}}{\dint{t}}\int\limits_{\Rset^n}\int\limits_\Rset \psi\at{p}\varrho\triple{t}{x}{p}\dint{p}\dint{x}&=\int\limits_{\Rset^n}\int\limits_\Rset \Bat{\nu^2\psi^{\prime\prime}\at{p}-\bat{H^\prime\at{p}-\si}\psi^\prime\at{p}}\varrho\triple{t}{x}{p}\dint{p}\dint{x}
\\&=\int\limits_{\Rset^n}\int\limits_\Rset \varrho\triple{t}{x}{p}\dint{p}\dint{x}=1\,.
\end{split}
\end{align}
By Variation of Constants we further demonstrate that $\psi$ 
satisfies 
\begin{align*}
\psi^\prime\at{p}=\frac{1}{\ga\at{p}}\at{c\,\ga\at{0}+\int\limits_0^p \frac{\ga\at{q}}{\nu^2}\dint{q}}
\end{align*}
with  $\ga$ as in \eqref{Eqn:Gibbs}, and conclude that there is precisely one choice of $c$, namely %
\begin{align*}
c=\frac{1}{\nu^2\bat{\ga\at{L}-\ga\at{0}}}\int\limits_0^L\ga\at{q}\dint{q}>0\,,
\end{align*}
such that $\psi^\prime$ is $L$-periodic. This implies %
\begin{align*}
\abs{\int\limits_{\Rset^n}\int\limits_\Rset \psi\at{p}\varrho\triple{t}{x}{p}\dint{p}\dint{x}-
\frac{\calP\at{t}}{L}\int\limits_0^L \psi^\prime\at{q}\dint{q}} \leq C
\end{align*} 
since $p\mapsto \psi\at{p} -p \,L^{-1} \int_0^L\psi^\prime\at{q}\dint{q}$ is bounded, and it remains to compute the integral of $\psi^\prime$ over $\ccinterval{0}{L}$. Inserting the ansatz 
\begin{align*}
\psi^\prime =:u= u_0+\nu^2u_1+\nu^4 u_2+\tdots
\end{align*}
into the differential equation \eqref{Prop:Ballistic.PEqn0} we verify
\begin{align*}
u_0\at{p}=\frac{1}{\si-H^\prime\at{p}}\,,\qquad u_1\at{p}=\frac{u_0^\prime\at{p}}{H^\prime\at{p}-\sigma}=-\frac{H^{\prime\prime}\at{p}}{\bat{\si- H^\prime\at{p}}^3}\,,
\end{align*}
and the claim follows after integrating \eqref{Prop:Ballistic.PEqn1} with respect to $t$.
\end{proof}
%
%
\section{Mass exchange in a double-well potential}
\label{sect:doublewell}
%
%
In this appendix we apply the asymptotic arguments from above to the case of a double-well potential as illustrated and described in Figure \ref{Fig:PlotDoubleWell}. For simplicity we restrict our considerations to the spatially homogeneous situation and study the Fokker-Planck equation
\begin{align}
\label{AppB:FP}
\tau\partial_t\varrho\pair{t}{p}=
\partial_p\bat{\nu^2\partial_p\varrho\pair{t}{p}+H^\prime\at{p}\varrho\pair{t}{p}}\,,
\end{align}
where the scaling law
\begin{align*}
\tau:= \om_0\,\om_-\exp\at{-\frac{h-}{\nu^2}}
\end{align*}
involves the curvature constants from Figure \ref{Fig:PlotDoubleWell} and is provided by Kramers formula.
For the latter we refer to \cite{Kra40,Ber13}, and to \cite{BEGK04,MS14} for generalization  to higher dimensions.  
\begin{figure}[ht!] %
\centering{ %
\includegraphics[width=0.45\textwidth]{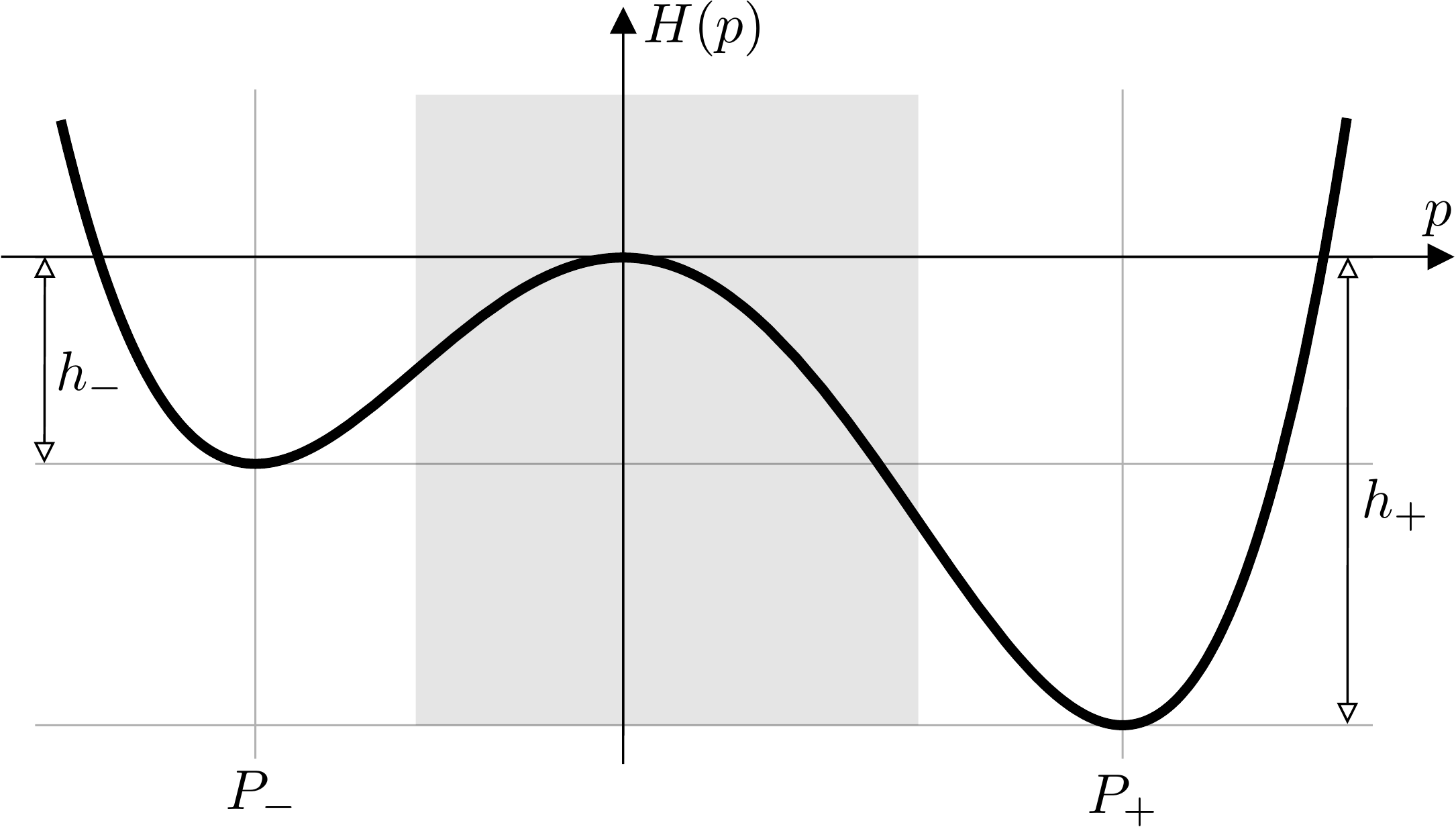}
} %
\caption{In this appendix, $H$ is a smooth double-well potential which grows at least quadratically at infinity, admits the normalized local maximum $H\at{0}=0$, and attains two local minima at $P_\pm$ with $P_-<0<P_+$. Moreover, we always assume $h_-\leq h_+$, where $h_{\pm}:=-H\at{P_\pm}>0$, and suppose that $H$ is non-degenerated according to $0>H^{\prime\prime}\at{0}=:-2\pi\om_0^2$ and
$0< H^{\prime\prime}\at{P_\pm}=:2\pi \om_\pm^2$.
} %
\label{Fig:PlotDoubleWell}
\end{figure}
%
%
\subsection{Limit dynamics and known proof strategies} %
%
As in the tilted case one expects that almost all mass of the system is -- at least for regular initial data and after a small transient time -- concentrated near the stable wells, i.e. in the vicinity of the local minima at $p=\pm P$. It is therefore natural to introduce the intervals (or `phases')
\begin{align*}
J_-:=\oointerval{-\infty}{0}\,,\qquad J_+:=\oointerval{0}{+\infty}
\end{align*}
and to define the partial masses by
\begin{align*}
m_\pm\at{t}:=\int\limits_{J_{\pm}}\varrho\pair{t}{p}\dint{p}\qquad\text{so that}\qquad  m_-\at{t}+m_+\at{t}=1\,.
\end{align*}
Using formal asymptotic analysis one easily shows -- see, e.g. \cite{Kra40,HNV12} for more details -- that the two phases exchange mass according to
\begin{align}
\label{AppB:Lim}
\dot{m}_+\at{t}=-\dot{m}_-\at{t}=\left\{\begin{array}{lcl}
m_-\at{t}&&\text{in the generic case with $h_->h_+$\,,}\\
m_-\at{t}-m_+\at{t}&&\text{in the symmetric case of $H\at{p}=H\at{-p}$\,,}\\
\end{array}\right.
\end{align}
and there exists several ways to derive the limit ODEs rigorously. The first one is to apply large deviation results to the underlying stochastic ODE, see for instance \cite{Ber13}, but alternative, PDE-analytic proofs have been given during the last decades by several authors in the framework of gradient flows. However, those proofs have so far been restricted to the symmetric situation with even function $H$ and require non-obvious modification in the general, asymmetric case.
\par
A first key observation -- both in the symmetric and the asymmetric case -- is that the Gibbs function 
\begin{align*}
\ga\at{p}:=\exp\at{-\nu^{-2}H\at{p}}
\end{align*}
is now integrable, so that \eqref{AppB:FP} admit the global equilibrium
\begin{align}
\label{AppB:EQ}
\ol{\ga}\at{p}:=\frac{\ga\at{p}}{\mu_-+\mu_+}\,,
\end{align}
where the constant $\mu_\pm$ will be computed below and ensure that $\int_\Rset\ol{\ga}\at{p}\dint{p}=1$. In terms of the relative density
\begin{align*}
u\pair{t}{p}:=\frac{\varrho\pair{t}{p}}{\ol{\ga}\at{p}}\,,
\end{align*}
the PDE \eqref{AppB:FP} reads
\begin{align}
\label{AppB:FP2}
\tau \ol{\ga}\at{p}\partial_t u\pair{t}{p}=\nu^2\partial_p \Bat{\ol{\ga}\at{p}\partial_p  u\pair{t}{p}}
\end{align}
and can be interpreted as scaled variant of the $\fspaceH^{0}_{\ol{\ga}}$-gradient flow to the $\fspaceH^1_{\ol{\ga}}$-energy of $u$, where the lower index indicates that the Sobolev spaces involve the weight function $\ol{\ga}$.  This Hilbert space formulation has -- in a slightly different setting -- been exploited in \cite{PSV10} for the rigorous derivation of the limit model in the symmetric case with even potential $H$. In particular, it has been shown that the quadratic metric tensor as well as the quadratic energy for $u$ --- which both depend on $\nu$ via $\ol\ga$ --- $\Gamma$-converge to limit objects that provide a linear gradient structure for the partial masses $\pair{m_-}{m_+}$. Finally, \cite{ET16} also passes to the limit $\nu\to0$ in \eqref{AppB:FP2} but exploits more elementary concepts instead of $\Gamma$-convergence.
\par
As already mentioned and shown in \cite{JKO97,JKO98}, the Fokker-Planck equation \eqref{AppB:FP} can also be regarded as the Wasserstein gradient flow to the energy
\begin{align*}
\calE\at{t}:=\int\limits_{\Rset}\at{\nu^2\varrho\pair{t}{p}\ln\bat{\varrho\pair{t}{p}}+H\at{p}\varrho\pair{t}{p}}\dint{p}\,,
\end{align*}
in the space of all probability measure on $\Rset$, and it is reasonable to ask whether one can also pass to the limit $\nu\to0$ in this non-flat setting with state-dependent metric tensor; cf. \cite{AGS08,Vil09} for the general theory of such gradient flows. A positive answer -- again in a slightly different setting --  has been given in \cite{HN11} and \cite{AMPSV12} using different concepts of evolutionary $\Gamma$-convergence, see especially \cite{AMPSV12} for a comparative discussion. However, both results are again restricted to the spatial case $h_-=h_+$ because otherwise $u$ cannot expected to be bounded independently of $\nu$. 
\par
In what follows we sketch an alternative derivation of  the limit models \eqref{AppB:Lim} which combines the dynamics of substitute masses with the a priori bounds for the Wasserstein dissipation, does not appeal to any notion of $\Gamma$-convergence, and covers both symmetric and asymmetric functions $H$.
%
%
\subsection{Substitute masses and passage to the limit}
%
%
In consistency with the case of a tilted periodic potential we define the scalar quantities
\begin{align*}
\mu_\pm:=\int\limits_{J_\pm}\ga\at{p}\dint{p}\,,\qquad  \eta:=\int\limits_{P_-}^{P_+}\frac{1}{\ga\at{p}}\dint{p}\,,\qquad \ka:=\frac{\mu_-}{\mu_+}\,,\qquad
\theta := \frac{\tau\mu_-\eta}{\nu^2}-1
\end{align*}
and introduce relative densities $\om^2_\pm:J_\pm\to\Rset$ by
\begin{align*}
w_\pm^2\pair{t}{p} : = \frac{\varrho\pair{t}{p}}{\ga_\pm\at{p}} \qquad \text{for}\quad p\in J_\pm \,,
\end{align*}
where 
\begin{align*}
\ga_\pm\at{p}:=\mu_\pm^{-1}\ga\at{p}\chi_{J_\pm}\at{p}
\end{align*}
represent normalized restrictions of $\ga$ to $J_\pm$. Moreover, choosing the moment weight $\psi$ according to
\begin{align}
\label{AppB:Mom}
\psi\at{P_-}=0\,,\qquad \psi^\prime\at{p}=\left\{\begin{array}{lcl}\D\frac{1}{\eta\ga\at{p}}&&\text{for $p\in\oointerval{P_-}{P_+}$\,,}\\0&&\text{else}\,,\end{array}\right.
\end{align} 
the substitute masses are given by
\begin{align*}
\widetilde{m}_-\at{t}:=\int\limits_{\Rset}\bat{1-\psi\at{p}}\varrho\pair{t}{p}\dint{p}\,,\qquad
\widetilde{m}_+\at{t}:=\int\limits_{\Rset}\psi\at{p}\varrho\pair{t}{p}\dint{p}
\end{align*}
and
\begin{align*}
\ol{m}_{\pm}\at{t}:=w_\pm\pair{t}{P_\pm}^2\,.
\end{align*}
We finally observe that
\begin{align*}
\varrho\pair{t}{p}=w_-\pair{t}{p}^2\ga_-\at{p}+w_+\pair{t}{p}^2\ga_-\at{p}
\end{align*}
as well as
\begin{align*}
m_\pm\at{t}=
\int\limits_{J_\pm} w_\pm^2\pair{t}{p}\ga_\pm\at{p}\dint{p},\,\qquad \qquad
m_-\at{t}+m_+\at{t}=
\widetilde{m}_-\at{t}+\widetilde{m}_+\at{t}=1
\end{align*}
hold by construction, and that the Wasserstein dissipation can be written as
\begin{align*}
\calD\at{t}=\int\limits_\Rset
\frac{\Bat{\partial_p\varrho\pair{t}{p}+ \nu^{-2}H^\prime\at{p}\varrho\pair{t}{p}}^2}{\varrho\pair{t}{p}}\dint{p}=4 D_-\at{t}+4D_+\at{t}
\end{align*}
with
\begin{align*}
 D_\pm\at{t}:=
\int\limits_{J_\pm}\bat{\partial_p w_\pm\pair{t}{p}}^2\ga_\pm\at{p}\dint{p}\,.
\end{align*}
Thanks to these definitions and employing the techniques from \S\ref{sect:AA} we establish the following results on the effective dynamics for $\nu\to0$.
\begin{proposition}[building blocks for the limit $\nu\to0$]
\label{AppB:Prop}
The following statements are satisfied for all sufficiently small $\nu>0$:
\begin{enumerate}
\item 
\ul{\emph{Elementary asymptotics:}} The scalar quantities fulfil 
\begin{align*}
\babs{\theta}+\abs{\eta\frac{\om_0}{\nu}-1}+\abs{\mu_\pm
\frac{\om_\pm}{\nu}\exp\at{-\frac{h_\pm}{\nu^2}}-1}
+\abs{\kappa
\frac{\om_-}{\om_+}\exp\at{\frac{h_+-h_-}{\nu^2}}-1}\leq C\nu^2\,.
\end{align*}
\item
\ul{\emph{Effective dynamics:}} The substitute masses evolve according to
 \begin{align}
\label{AppB:Dyn}
\pm\at{1+\theta}\frac{\dint}{\dint{t}}\widetilde{m}_\pm\at{t}=\ol{m}_-\at{t}-\kappa\ol{m}_+\at{t}\,.
\end{align}
\item \ul{\emph{Dissipation bounds error:}} We have
\begin{align*}
\babs{\widetilde{m}_\pm\at{t}-m_{\pm}\at{t}}+\babs{\ol{m}_\pm\at{t}-m_{\pm}\at{t}}\leq C\abs{\tau^{-1/2}\nu^2\calD\at{t}+\tau^{1/2}\nu^{-2}}\,.
\end{align*}
\item \ul{\emph{Energy balance:}}
The total Wasserstein dissipation is bounded by 
\begin{align*}
\int\limits_0^\infty \calD\at{t}\dint{t}\leq C\tau\nu^{-4}\bat{C+\calE\at{0}}\,.
\end{align*}
\end{enumerate}
Here, the constant $C$ depends on $H$ but not on $\nu$.
\end{proposition}
\begin{proof}
The first three assertions can be derived  analogously to the proofs of Lemma \ref{Lem:AsympIntegrals}, Proposition \ref{Prop:SubstituteMasses}, and Proposition \ref{Prop:ErrorBounds}. The justification of the fourth claim is even simpler than in \S\ref{sect:AA} because the energy $\calE$ is now globally bounded from below due to the existence of the global minimizer $\ol\ga$ from \eqref{AppB:EQ}. In particular, we have
\begin{align*}
\calE\at{t}\geq \int\limits_\Rset\ol{\ga}\at{p}\Bat{\nu^2\ln\bat{\ol{\ga}\at{p}}+H\at{p}}\dint{p}\geq -\nu^2\ln\at{\mu_-+\mu_+}\geq -C
\end{align*}
thanks to $\mu_\pm\sim \exp\at{-h_\pm/\nu^2}$.
\end{proof}
Proposition \ref{AppB:Prop} allow us to pass to the limit $\nu\to0$ similarly to Theorem \ref{Thm:Convergence}, i.e., by means of functions $\breve{m}_\pm$ which solve the limit ODEs and attain the same initial values as $\widetilde{m}_\pm$. The outcome can informally stated as follows.
\begin{corollary}[limit dynamics for $\nu\to0$]
For sufficiently nice initial data, the asymptotic mass exchange is governed by
\begin{align*}
\pm \dot{m}_\pm\at{t}=m_-\at{t}
\end{align*}
for $h_->h_+$ and by 
\begin{align*}
\pm \dot{m}_\pm\at{t}=m_-\at{t}-\kappa m_+\at{t}
\end{align*}
in the non-generic case of $h_-=h_+$,  where $\kappa=\om_+/\om_-$.
\end{corollary}
We finally emphasize that the primitive of $1/\ga$, which defines the moment weight $\psi$ in \eqref{AppB:Mom}, features prominently also in \cite{PSV10,HN11,AMPSV12,HNV14,ET16} but it seems that this function has never been used before to establish a dynamical identity like \eqref{AppB:Dyn}.
%

%


\begin{thebibliography}{AMP{\etalchar{+}}12}

\bibitem[AGS08]{AGS08}
L.~Ambrosio, N.~Gigli, and G.~Savar\'e.
\newblock {\em Gradient flows in metric spaces and in the space of probability
  measures}.
\newblock Lectures in Mathematics ETH Z\"urich. Birkh\"auser Verlag, Basel,
  second edition, 2008.

\bibitem[AMP{\etalchar{+}}12]{AMPSV12}
S.~Arnrich, A.~Mielke, M.A. Peletier, G.~Savar\'e, and M.~Veneroni.
\newblock Passing to the limit in a {W}asserstein gradient flow: from diffusion
  to reaction.
\newblock {\em Calc. Var. Partial Differential Equations}, 44(3-4), 2012.

\bibitem[BEGK04]{BEGK04}
A.~Bovier, M.~Eckhoff, V.~Gayrard, and M.~Klein.
\newblock Metastability in reversible diffusion processes. {I}. {S}harp
  asymptotics for capacities and exit times.
\newblock {\em J. Eur. Math. Soc. (JEMS)}, 6(4):399--424, 2004.

\bibitem[Ber13]{Ber13}
N.~Berglund.
\newblock Kramers' law: validity, derivations and generalisations.
\newblock {\em Markov Process. Related Fields}, 19(3):459--490, 2013.

\bibitem[BO99]{BO99}
C.M. Bender and S.A. Orszag.
\newblock {\em Advanced mathematical methods for scientists and engineers.
  {I}}.
\newblock Springer-Verlag, New York, 1999.
\newblock Asymptotic methods and perturbation theory, Reprint of the 1978
  original.

\bibitem[CY15]{CY15}
L.~Cheng and N.K. Yip.
\newblock The long time behavior of {B}rownian motion in tilted periodic
  potentials.
\newblock {\em Phys. D}, 297:1--32, 2015.

\bibitem[ET16]{ET16}
L.C. Evans and P.R. Tabrizian.
\newblock Asymptotics for scaled {K}ramers-{S}moluchowski equations.
\newblock {\em SIAM J. Math. Anal.}, 48(4):2944--2961, 2016.

\bibitem[Fri64]{Fri64}
A.~Friedman.
\newblock {\em Partial differential equations of parabolic type}.
\newblock Prentice-Hall, Inc., Englewood Cliffs, N.J., 1964.
\newblock reprinted 2008 by Dover Publications.

\bibitem[HN11]{HN11}
M.~Herrmann and B.~Niethammer.
\newblock Kramers' formula for chemical reactions in the context of
  {W}asserstein gradient flows.
\newblock {\em Commun. Math. Sci.}, 9(2):623--635, 2011.

\bibitem[HNV12]{HNV12}
M.~Herrmann, B.~Niethammer, and J.J.L. Vel\'azquez.
\newblock Kramers and non-{K}ramers phase transitions in many-particle systems
  with dynamical constraint.
\newblock {\em Multiscale Model. Simul.}, 10(3):818--852, 2012.

\bibitem[HNV14]{HNV14}
M.~Herrmann, B.~Niethammer, and J.J.L. Vel\'{a}zquez.
\newblock Rate-independent dynamics and {K}ramers-type phase transitions in
  nonlocal {F}okker-{P}lanck equations with dynamical control.
\newblock {\em Arch. Ration. Mech. Anal.}, 124(3):803--866, 2014.

\bibitem[HP08]{HP08}
M.~Hairer and G.~A. Pavliotis.
\newblock From ballistic to diffusive behavior in periodic potentials.
\newblock {\em J. Stat. Phys.}, 131(1):175--202, 2008.

\bibitem[JKO97]{JKO97}
R.~Jordan, D.~Kinderlehrer, and F.~Otto.
\newblock Free energy and the {F}okker-{P}lanck equation.
\newblock {\em Phys. D}, 107(2-4):265--271, 1997.
\newblock Landscape paradigms in physics and biology (Los Alamos, NM, 1996).

\bibitem[JKO98]{JKO98}
R.~Jordan, D.~Kinderlehrer, and F.~Otto.
\newblock The variational formulation of the {F}okker-{P}lanck equation.
\newblock {\em SIAM J. Math. Anal.}, 29(1):1--17, 1998.

\bibitem[Kra40]{Kra40}
H.A. Kramers.
\newblock Brownian motion in a field of force and the diffusion model of
  chemical reactions.
\newblock {\em Physica}, 7:284--304, 1940.

\bibitem[LKSG01]{LKSG01}
B.~Lindner, M.~Kostur, and L.~Schimansky-Geier.
\newblock Optimal diffusive transport in a tilted periodic potential.
\newblock 1(1):R25--R39, 2001.

\bibitem[LPK13]{LPK13}
J.C. Latorre, G.A. Pavliotis, and P.R. Kramer.
\newblock Corrections to {E}instein's relation for {B}rownian motion in a
  tilted periodic potential.
\newblock {\em J. Stat. Phys.}, 150(4):776--803, 2013.

\bibitem[MS14]{MS14}
G.~Menz and A.~Schlichting.
\newblock Poincar\'e and logarithmic {S}obolev inequalities by decomposition of
  the energy landscape.
\newblock {\em Ann. Probab.}, 42(5):1809--1884, 2014.

\bibitem[MZ17]{MZ17}
L.~Michel and M.~Zworski.
\newblock A semiclassical approach to the {K}ramers-{S}moluchowski equation.
\newblock preprint arXiv:1703.07460, 2017.

\bibitem[PSV10]{PSV10}
M.A. Peletier, G.~Savar\'e, and M.~Veneroni.
\newblock From diffusion to reaction via {$\Gamma$}-convergence.
\newblock {\em SIAM J. Math. Anal.}, 42(4), 2010.

\bibitem[Ris89]{Ris89}
H.~Risken.
\newblock {\em The {F}okker-{P}lanck equation}, volume~18 of {\em Springer
  Series in Synergetics}.
\newblock Springer-Verlag, Berlin, second edition, 1989.
\newblock Methods of solution and applications.

\bibitem[RVL{\etalchar{+}}02]{ReiEtAl02}
P.~Reimann, C.~{Van~den~Broeck}, H.~Linke, P.~H\"anggi, J.M. Rubi, and
  A.~P\'erez-Madrid.
\newblock Diffusion in tilted periodic potentials: Enhancement, universality,
  and scaling.
\newblock {\em Phys. Rev. E}, 65(3):031104:1--16, 2002.

\bibitem[SL10]{SL10}
J.M. Sancho and A.M. Lacasta.
\newblock The rich phenomenology of brownian particles in nonlinear potential
  landscapes.
\newblock {\em Eur. Phys. J. Spec. Top.}, 187(1):49--62, 2010.

\bibitem[Vil09]{Vil09}
C.~Villani.
\newblock {\em Optimal transport}, volume 338 of {\em Grundlehren der
  Mathematischen Wissenschaften [Fundamental Principles of Mathematical
  Sciences]}.
\newblock Springer-Verlag, Berlin, 2009.
\newblock Old and new.

\end{thebibliography}

\newcommand{\etalchar}[1]{$^{#1}$}

\end{document}